\documentclass{article}
\usepackage[a4paper,top=3cm,bottom=3cm,inner=3cm,outer=3cm]{geometry}
\usepackage[utf8]{inputenc}
\usepackage{amsmath,amsthm,amsfonts,amssymb}

\usepackage{color}
\definecolor{myurlcolor}{rgb}{0.6,0,0}
\definecolor{mycitecolor}{rgb}{0,0,0.9}
\definecolor{myrefcolor}{rgb}{0,0,0.9}
\usepackage[colorlinks,citecolor=darkgreen,urlcolor=gray,final,hyperindex,linktoc=page]{hyperref} 
\hypersetup{colorlinks,
linkcolor=myrefcolor,
citecolor=mycitecolor,
urlcolor=myurlcolor}
\hypersetup{
    colorlinks=true,
    linkcolor=blue,
    urlcolor=red,
    linktoc=all
}

\usepackage{adjustbox}

\usepackage{tikz, tikz-cd}
\usetikzlibrary{positioning}
\definecolor {processblue}{cmyk}{0.9,0.5,0,0}

\tikzstyle{simple}=[-,line width=2.000]
\tikzstyle{arrow}=[-,postaction={decorate},decoration={markings,mark=at position .5 with {\arrow{>}}},line width=1.100]
\pgfdeclarelayer{edgelayer}
\pgfdeclarelayer{nodelayer}
\pgfdeclarelayer{bg}
\pgfsetlayers{bg,edgelayer,nodelayer,main}
\tikzstyle{none}=[inner sep=-1pt]
\tikzstyle{species}=[circle,fill=none,draw=black,scale=1.0]
\tikzstyle{transition}=[rectangle,fill=none,draw=black,scale=1.15]
\tikzstyle{empty}=[circle,fill=none, draw=none]
\tikzstyle{inputdot}=[circle,fill=black,draw=black, scale=.5]
\tikzstyle{dot}=[circle,fill=black,draw=black]
\tikzstyle{bounding}=[circle,dashed, fill=none,draw=black, scale=9.00]
\tikzstyle{simple}=[-,draw=black,line width=1.000]
\tikzstyle{inarrow}=[-,draw=black,postaction={decorate},decoration={markings,mark=at position .5 with {\arrow{>}}},line width=1.000]
\tikzstyle{tick}=[-,draw=black,postaction={decorate},decoration={markings,mark=at position .5 with {\draw (0,-0.1) -- (0,0.1);}},line width=1.000]
\tikzstyle{inputarrow}=[->,draw=black, shorten >=.05cm]

\tikzset{main node/.style={circle,fill=blue!20,draw,minimum size=1cm,inner sep=0pt},}

\tikzstyle{construct}=[fill=white, draw=black, shape=circle]
\tikzstyle{universal}=[fill=black, draw=black, shape=circle]

\newcommand{\2}{\mathbf 2}

\newcommand{\GMV}{\G_{M,\V}}

\newcommand{\Aut}{\mathrm{Aut}}

\newcommand{\define}[1]{{\bf \boldmath{#1}}}

\newcommand{\catname}[1]{\mathcal{#1}}
\newcommand{\namedcat}[1]{\mathsf{#1}}
\newcommand{\bink}[1]{\binom{#1}{k}}

\newcommand{\C}{\catname C} 
\newcommand{\Alg}{\namedcat{Alg}}
\newcommand{\Cat}{\namedcat{Cat}}

\newcommand{\CMon}{\namedcat{CMon}}
\newcommand{\FinInj}{\namedcat{FinInj}}
\newcommand{\IMon}{\namedcat{IMon}}
\newcommand{\Inj}{\namedcat{Inj}}
\newcommand{\GMon}{\namedcat{GMon}}
\newcommand{\Mon}{\namedcat{Mon}}
\newcommand{\NM}{\namedcat{NetMod}}
\newcommand{\Quiv}{\namedcat{Quiv}} 
\renewcommand{\S}{\namedcat S}
\newcommand{\Set}{\namedcat{Set}}
\newcommand{\sGrph}{\namedcat{SimpGph}} 

\newcommand{\V}{\catname V} 

\newcommand{\KG}{\mathrm{KG}} 

\newcommand{\B}{\mathbb B}

\newcommand{\G}{\Gamma}
\newcommand{\n}{\mathbf n}
\newcommand{\m}{\mathbf m}

\newcommand{\SG}{\mathrm{SG}}
\renewcommand{\O}{\mathrm O}

\newcommand{\N}{\mathbb N}
\newcommand{\R}{\mathbb R}
\newcommand{\maps}{\colon}

\DeclareMathOperator{\colim}{colim}

\theoremstyle{plain}
\newtheorem{thm}{Theorem}
\newtheorem{lem}[thm]{Lemma}
\newtheorem{prop}[thm]{Proposition}

\newtheorem{defn}[thm]{Definition}
\newtheorem{expl}[thm]{Example}
\newtheorem*{defn*}{Definition}
\newtheorem{expl*}[thm]{Example}
\newtheorem{question}[thm]{Question}
\newtheorem*{thm*}{Theorem}
\newtheorem*{lem*}{Lemma}
\newtheorem*{prop*}{Proposition}
\newtheorem*{cor*}{Corollary}
\theoremstyle{remark}

\title{Noncommutative Network Models}
\author{Joe Moeller}
\date{\today}

\begin{document}

\maketitle

\begin{abstract}
    Network models, which abstractly are given by lax symmetric monoidal functors, are used to construct operads for modeling and designing complex networks. Many common types of networks can be modeled with simple graphs with edges weighted by a monoid. A feature of the ordinary construction of network models is that it imposes commutativity relations between all edge components. Because of this, it cannot be used to model networks with bounded degree. In this paper, we construct the free network model on a given monoid, which can model networks with bounded degree. To do this, we generalize Green's graph products of groups to pointed categories which are finitely complete and cocomplete.
\end{abstract}


\section{Network Models}
\label{sec:netmods}

One way to combine two graphs is to identify the vertices of one with some of the vertices of the other in a one-to-one way, then gluing the two graphs together at the identified vertices. 
We can decompose such a combination into a sequence of simpler operations of three types: disjoint union of any two graphs, gluing two graphs together which have the same vertex set, which we call \emph{overlay}, and permutation of vertices. 
In previous work, \emph{network models} were introduced to formally encode these operations \cite{netmods}. The algebras of a network operad can serve as tools for designing complex multi-agent networks. 
Network operads are constructed from network models, which are certain symmetric lax monoidal functors.
There is a functorial construction of a network model from a monoid, which we call the \emph{ordinary network model for weighted graphs}. 
In this paper, we provide a different construction in order to realize a larger class of networks as algebras of network operads, which we call the \emph{free varietal network model for weighted graphs}. In Section \ref{sec:commitment}, we give an example of a family of networks which cannot form an algebra for any ordinary network model for weighted graphs, but does for a varietal one.

The reader is assumed to be familiar with basic notions from category theory \cite{CWM}, especially symmetric monoidal categories and lax symmetric monoidal functors \cite{braided}. Let $\S$ be the \define{symmetric groupoid}, i.e.\ the category with objects $\n = \{1, \dots, n\}$ (including the empty set for $\mathbf{0}$) and bijections for morphisms. Let $\Mon$ denote the category of monoids. A \define{one-colored network model} is a symmetric lax monoidal functor 
\[
    (F, \Phi) \maps (\S, +) \to (\Mon, \times)
\] 
where $\Phi$ is the \emph{laxator} of $F$, i.e.\ a natural transformation with components 
\[
    \Phi_{x,y} \maps F \n \times F \m \to F(\n + \m).
\] 
We call the monoids $F(\n)$ the \define{constituent monoids} of the network model $F$. There is a more general notion of network model which replaces the category $\S$ with a free symmetric monoidal category. We do not consider this generalization here, so we always mean a one-colored network model when we say network model. Let $\NM$ denote the category of network models with monoidal natural transformations as morphisms.

Essentially, a network model is a family of monoids $\{M_n\}_{n \in \N}$ each with a group action of the corresponding symmetric group $S_n$, such that the product of any two embed into the one indexed by the sum of their indices equivariantly, i.e.\ in a way which respects the group action: $M_m \times M_n \hookrightarrow M_{m+n}$ .

\begin{expl}
    For a set $X$ and $k \in \N$, let $\bink{X}$ denote the set of $k$-element subsets of $X$. 
    A simple graph on $\n$ is a subset of $\binom{\n}{2}$. The set $\n$ is the set of vertices, and each two-element subset is an edge.
    Let $\SG(\n)$ be the monoid whose underlying set is $2^{\binom{\n}{2}}$, the set of all simple graphs with vertex set $\n$, and whose monoid operation is union.
    \[\scalebox{0.8}
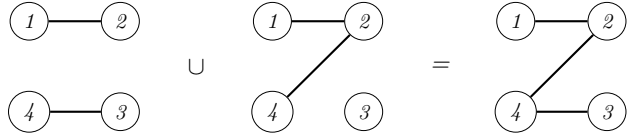
{
    \begin{tikzpicture}
    	\begin{pgfonlayer}{nodelayer}
    	    \node [style=none, scale = 1.2] () at (5,2) {$\cup$};
    	    \node [style=none, scale = 1.2] () at (9,2) {=};
    		\node [style=species]  (1) at (3.75, 2.75) {2};
    		\node [style=species]  (2) at (2.25, 2.75) {1};
    		\node [style=species]  (3) at (2.25, 1.25) {4};
    		\node [style=species]  (4) at (3.75, 1.25) {3};
    		\node [style=species]  (5) at (7.75, 2.75) {2};
    		\node [style=species]  (6) at (6.25, 2.75) {1};
    		\node [style=species]  (7) at (6.25, 1.25) {4};
    		\node [style=species]  (8) at (7.75, 1.25) {3};
    		\node [style=species]  (9) at (11.75, 2.75) {2};
    		\node [style=species]  (10) at (10.25, 2.75) {1};
    		\node [style=species]  (11) at (10.25, 1.25) {4};
    		\node [style=species]  (12) at (11.75, 1.25) {3};
    	\end{pgfonlayer}
    	\begin{pgfonlayer}{edgelayer}
    		\draw [style=simple] (2) to (1);
    		\draw [style=simple] (3) to (4);
    		\draw [style=simple] (6) to (5);
    		\draw [style=simple] (5) to (7);
    		\draw [style=simple] (10) to (9);
    		\draw [style=simple] (9) to (11);
    		\draw [style=simple] (11) to (12);
    	\end{pgfonlayer}
    \end{tikzpicture}}\]
    The group $S_n$ acts on $\SG(\n)$ by permuting the vertices, so we have a functor $\SG \maps \S \to \Mon$. 
    Disjoint union of graphs defines a family of monoid homomorphisms $\sqcup \maps \SG(\m) \times \SG(\n) \to \SG(\m+\n)$. 
    These maps form a natural transformation which acts as the laxator, making $(\SG, \sqcup) \maps (\S, +) \to (\Mon, \times)$ a lax symmetric monoidal functor, thus a network model.
\end{expl}

The simple graph network model is the motivating example for network models. Since the operation in the constituent monoids of $\SG$ are defined by union, we use $\cup$ to denote the operation in all the constituent monoids for all network models. Since the elements of the monoids $\SG(\n)$ are graphs, for a general network model $F$, the elements of the monoids $F(\n)$ are called \define{$F$-networks}.
Since a network model is a lax symmetric monoidal functor, we can apply a symmetric monoidal variant of the Grothendieck construction to obtain a symmetric monoidal category. From this, we obtain the \emph{underlying operad}. 
The main result on network models is that given any network model $F \maps \S \to \Mon$, there is an $\N$-typed operad $O_F$ with $F$-networks as its operations \cite{netmods}.

This gives us a method of obtaining operads that are useful for modeling complex networks of various sorts. We start with a suitable network model, and then apply this theorem to obtain an operad. 
This leads to the question: which kinds of network can be described by network models?

One large class of network models comes from monoids \cite{netmods}.  Indeed, for any monoid $M$ there is a network model for which the networks are simple graphs weighted by $M$.
In the simple graphs example, $\SG(\n)$ is the monoid whose underlying set was $2^{\binom{\n}{2}}$ and whose monoid operation was union. 
This could be said more succinctly by letting $\B$ be the boolean monoid, $(\{T, F\}, \vee)$, and then defining $\SG(\n) = \B^{\binom{\n}{2}}$. 
We interpret an element of $\SG(\n)$, which is a function $g \maps \binom{\n}{2} \to \B$, as a graph on $\n$ with an edge between $u \in \n$ and $v \in \n$ if and only if $g(\{u,v\}) =T$.
For a given monoid, let $\G_M$ be the network model defined by $\G_M(\n) = M^{\binom{\n}{2}}$, and everything else defined as with simple graphs. 
As with $\SG(\n)$, an element of $\G_M(\n)$ should be thought of as a simple graph, but with edges weighted by values in $M$. This construction is functorial: 

\begin{thm}[c.f.\ \cite{netmods}, Thm.\ 22]
\label{thm:ordinary}
    There is a functor $\G \maps \Mon \to \NM$ sending any monoid $M$ to the network model $\G_M$ and any homomorphism of monoids $f \maps M \to M'$ to the morphism of network models $\G_f \maps \G_M \to \G_{M'}$. The network model $\G_M$ is called the \define{ordinary network model for $M$-weighted graphs} or the \define{ordinary $M$ network model}.
\end{thm}

This construction is designed to model networks which carry information on the edges. 
For example, with $\N$ a monoid under addition, $\G_\N$ is a network model for loopless undirected multigraphs where overlaying is given by adding the number of edges. A similar example is $\G_\B = \SG$. There is a monoid homomorphism $\N \to \B$ which sends all but $0$ to $T$. This induces a map of network models $\G_\N \to \G_\B$. 
Essentially this map reduces the information of a graph from the number of connections between each pair of vertices to just the existence of any connection.

\begin{expl}[\bf Algebra for range-limited communication]
    Consider a communication network where each node represents a boat and an edge between two nodes represents a working communication channel between the corresponding boats. Some forms of communication are restricted by the distance between those communicating. Assume that there is a known maximal distance over which our boats can communicate. Networks of this sort form an algebra of the simple graphs operad in the following way.
    
    Let $(X,d)$ be a metric space, and $0 \leq L \in \R$. Our boats will be located at points in this space. The operad $\O_\SG$ has an algebra $(A_{d,L}, \alpha)$ defined as follows. The set $A_{d,L}(\n)$ is the set of pairs $(h,f)$ where $h \in \SG(\n)$ is a simple graph and $f \maps \n \to X$ is a function such that if $\{v_1,v_2\}$ is an edge in $g$ then $d(f(v_1),f(v_2)) \leq L$. The number $L$ represents the maximal distance over which the boat's communication channels operate. Notice that this condition does not demand that all connections within range must be made. An operation $(\sigma, g)\in \O_\SG(\n_1, \dots, \n_k; \n)$ acts on a $k$-tuple $(h_i, f_i) \in A_{d,L}(\n_i)$ by 
    \[
        \alpha(\sigma, g)((h_1, f_1), \dots, (h_k, f_k)) = (g \cup \sigma(h_1 \sqcup \dots \sqcup h_k), f_1 \sqcup \dots \sqcup f_k).
    \] 
    Elements of this algebra are simple graphs in the space $X$ with an upper limit on edge lengths. When an operation acts on one of these, it tries to put new edges into the graph, but fails to when the range limit is exceeded \cite{netmods}.
\end{expl}

A characteristic of the construction given in Theorem \ref{thm:ordinary} is that elements of the resulting monoids that correspond to different edges automatically commute with each other. 
For example, for a monoid $M$, the fourth constituent monoid of the ordinary $M$ network model is $\G_M(4) = M^6$.
Then the element $(m_1, 0, 0, 0, 0, 0)$ represents a graph with one edge with weight $m_1 \in M$, the element $(0, m_2, 0, 0, 0, 0)$ represents a graph with a different edge with weight $m_2 \in M$, and 
\begin{align*}
    (m_1, 0, 0, 0, 0, 0) \cup (0, m_2, 0, 0, 0, 0) 
    &= (m_1, m_2, 0, 0, 0, 0) 
    \\&= (0, m_2, 0, 0, 0, 0) \cup (m_1, 0, 0, 0, 0, 0).
\end{align*}

This commutativity between edges means that networks given by ordinary network models cannot record information about the order in which edges were added to it. 
The ability to record such information about a network is desirable, for example, if one wishes to model networks which have a limit on the number of connections each agent can make to other agents. 

The \define{degree} of a vertex in a simple graph is the number of edges which include that vertex. 
The \define{degree} of a graph is the maximum degree of its vertices.
A graph is said to have \define{degree bounded by $k$}, or simply \define{bounded degree}, if the degree of each vertex is less than or equal to $k$.
Let $B_k(\n)$ denote the set of networks with $\n$ vertices and degree bound $k$. 
One might guess that the family of such networks could form an algebra for the simple graphs operad. 

\begin{question}
\label{question}
    Does the collection of networks of bounded degree form an algebra of a network operad? If so, is there such an algebra which is useful in applications?
\end{question}

Specifically, can networks of bounded degree form an algebra of $\O_\SG$, the simple graph operad? Setting two graphs next to each other will not change the degree of any of the vertices. Overlaying them almost definitely will, which makes defining an action of $\SG(\n)$ on $B_k(\n)$ less obvious. 

Ordinary network models are not sufficient to model this type of network because the graph monoids it produced could not remember the order that edges were added into a network. Even if $M$ is a noncommutative monoid, since $\G_M$ is a product of several copies of $M$, one for each pair of vertices, it cannot distinguish the order that two different edges touching $v_1$ were added to a network if their other endpoints are different. 

Instead of taking the product of $\binom{n}{2}$ copies of $M$, we consider taking the coproduct, so as not to impose any commutativity relations between the edges.
Since the lax structure map $\sqcup \maps F(\m) \times F(\n) \to F(\m+\n)$ associated to a network model $F \maps \S \to \Mon$ must be a monoid homomorphism, then \[(a \sqcup b) \cup (c \sqcup d) = (a \cup c) \sqcup (b \cup d).\] In particular, if we let $\emptyset$ denote the the identity of $F(\n)$ for any $\n$, then 
\begin{align*}
    (a \sqcup \emptyset) \cup (\emptyset \sqcup b) 
    &= (a \cup \emptyset) \sqcup (\emptyset \cup b)
    \\&= (\emptyset \cup a) \sqcup (b \cup \emptyset)
    \\&= (\emptyset \sqcup b) \cup (a \sqcup \emptyset).
\end{align*}
This is reminiscent of the Eckmann-Hilton argument, but notice that the domains of the operations $\cup$ and $\sqcup$ are not the same. This equation says that elements which correspond to disjoint edges must commute with each other. Simply taking the coproduct of $\binom{n}{2}$ copies of $M$ cannot give the constituent monoids of a network model.

For a collection of monoids $\{M_i\}_{i \in I}$, elements of the product monoid which come from different components always commute with each other. In the coproduct, they never do. A \emph{graph product} (in the sense of Green \cite{Green}) of such a collection allows one to impose commutativity between certain components and not others by indicating such relations via a simple graph. The calculation above shows that the constituent monoids of a network model must satisfy certain partial commutativity relations. We use graph products to construct a family of monoids with the right amount of commutativity to both answer question \ref{question} and satisfy the conditions of being a network model. The following theorems are proven in Section \ref{sec:funnetmod}.

\begin{thm*}
    The functor $\NM \to \Mon$ defined by $F \mapsto F(\2)$ has a left adjoint $\G_{-,\Mon} \maps \Mon \to \NM$.
\end{thm*}

The fact that this construction is a left adjoint tells us that the network models constructed are ones in which the only relations that hold are those that follow from the defining axioms of network models. 

A \emph{variety} of monoids is the class of all monoids satisfying a given set of identities. 
For example, $\Mon$ has subcategories $\CMon$ of commutative monoids and $\GMon$ of graphic monoids which are varieties of monoids satisfying the equations $ab=ba$ and $aba=ab$ respectively. Given a variety of monoids $\V$, let $\NM_\V$ be the subcategory of $\NM$ consisting of $\V$-valued network models. We recreate graph products in varieties of monoids to obtain a more general result.

\begin{thm*}
    The functor $\NM_V \to \V$ defined by $F \mapsto F(\2)$ has a left adjoint $\G_{-,\V} \maps \V \to \NM_V$.
\end{thm*}

In particular, if $\V = \CMon$, since products and coproducts are the same in $\CMon$, the ordinary $M$ network model and the $\CMon$ varietal $M$ network model are also the same. Note that this does not indicate that $\G_{-,\V}$ is a complete generalization of $\G_-$ from Theorem \ref{thm:ordinary}, since $\G_M$ is not an example of $\G_{-,\V}$ when $M$ is not commutative. 

The ordinary construction for a network model given a monoid $M$ has constituent monoids given by finite cartesian powers of $M$. To include the networks described in Question \ref{question} into the theory of network models, we must construct a network model from a given monoid which does not impose as much commutativity as the ordinary construction does, specifically among elements corresponding to different edges. The first attempt at a solution is to use coproducts instead of products. However, in this section we saw that we cannot create the constituent monoids of a network model simply by taking them to be coproducts of $M$ instead of products. There must be some commutativity between different edges, specifically between edges which do not share a vertex. 

Given a monoid $M$, we want to create a family of monoids indexed by $\N$, the $n$th of which looks like a copy of $M$ for each edge in the complete graph on $\n$, has minimal commutativity relations between these edge components, but does have commutativity relations between disjoint edges. Partial commutativity like this can be described with Green's graph products, which we describe in Section \ref{sec:graphprod}. The type of graph which describes disjointness of edges in a graph as we need is called a \emph{Kneser graph}, which we describe in Section \ref{sec:kneser}. Besides concerning ourselves with relations between edge components, sometimes we also want the constituent monoids in a network model to obey certain relations which $M$ obeys. In Section \ref{sec:varmon} we describe \emph{varieties of monoids} and a construction which produces monoids in a chosen variety. In Section \ref{sec:funnetmod} we prove this construction is functorial, and in Section \ref{sec:commitment} we use this construction to give a positive answer to Question \ref{question}.

\section{Graphs}

This section is dedicated to constructing the constituent monoids for the network models we want. In this section there are two different ways that graphs are being used. It is important that the reader does not get these confused. One way is the graphs which are elements of the constituent monoids of the network models we are constructing. The other way we use graphs is to index the \emph{Green product} (which we define in \ref{sec:graphprod}) to describe commutativity relations in the constituent monoids of the network models we are constructing. 

A network model is essentially a family of monoids with properties similar to the simple graphs example, so we think of the elements of these monoids as graphs, and we think of the operation as overlaying the graphs. 
These monoids have partial commutativity relations they must satisfy, as we see in Section \ref{sec:netmods}.
The graphs we use in the Green product, the Kneser graphs, are there to describe the partial commutativity in the constituent monoids. 

\subsection{Green Products}
\label{sec:graphprod}

Given a family of monoids $\{M_v\}_{v\in V}$ indexed by a set $V$, there are two obvious ways to combine them to get a new monoid, the product and the coproduct.
From an algebraic perspective, a significant difference between these two is whether or not elements that came from different components commute with each other. In the product they do. In the coproduct they do not. 
\emph{Green products}, or commonly \emph{graph products}, of groups were introduced in 1990 by Green \cite{Green}, and later generalized to monoids by Veloso da Costa \cite{Veloso}.
The idea provides something of a sliding scale of relative commutativity between components. We follow \cite{Fountain} in the following definitions.

By a \define{simple graph} $G=(V, E)$, we mean a set $V$ which we call the set of vertices, and a set $E \subseteq \binom{V}{2}$, which we call the set of edges. A \define{map of simple graphs} $f \maps (V, E) \to (V', E')$ is
a function $f \maps V \to V'$ such that if $\{u, v\} \in E$ then $\{f(u), f(v)\} \in E'$. 
Let $\sGrph$ denote the category of simple graphs and maps of simple graphs.

For a set $V$, a family of monoids $\{M_v\}_{v\in V}$, and a simple graph $G = (V, E)$, the \define{$G$ Green product} (or simply \define{Green product} when unambiguous) of $\{M_v\}_{v\in V}$, denoted $G(M_v)$, is 
\[
    G(M_v) = \left( \coprod_{v\in V} M_v \right) /R_G
\]
where $R_G$ is the congruence generated by the relation 
\[
    \{ (m n, n m) |\, m \in M_v, n\in M_u, u, v \text{ are adjacent in }G \}
\]
where the operation in the free product is denoted by concatenation. 
If $G$ is the complete graph on $n$ vertices, then $G(M_v) \cong \prod M_v$. If $G$ is the $n$-vertex graph with no edges, then $G(M_v) \cong \coprod M_v$. 

We call each $M_v$ a \define{component} of the Green product.
Elements of $G(M_v)$ are written as \define{expressions} as in the free product, $m^{v_1}_1\dots m^{v_k}_k \in G(M_v)$ where the superscript indicates that $m_i \in M_{v_i}$.
We often consider Green products of several copies of the same monoid, so this notation allows one to distguish elements coming from different components of the product, even if they happen to come from the same monoid. 
The intention and result of the imposed relations is that for an expression $m^{v_1}_1\dots m^{v_k}_k$ of an element, if there is an $i$ such that
$\{v_i, v_{i+1}\} \in E$, then we can rewrite the expression by replacing $m^{v_i}_i m^{v_{i+1}}_{i+1}$ with $m^{v_{i+1}}_{i+1} m^{v_i}_i$.
This move is called a \define{shuffle}, and two expressions are called \define{shuffle equivalent} if one can be obtained from the other by a sequence of shuffles. 
An expression $m^{v_1}_1\dots m^{v_k}_k$ is \define{reduced} if whenever $i<j$ and $v_i = v_j$, there exists $l$ with $i<l<j$ and $\{v_i, v_l\} \notin E$.
If two reduced expressions are shuffle equivalent, they are clearly expressions of the same element. The converse is also true.

\begin{thm}[\cite{Fountain}, Thm.\ 1.1]
\label{thm:shuffle}
    Every element of $M$ is represented by a reduced expression. Two reduced expressions represent the same element of $M$ if and only if they are shuffle equivalent.
\end{thm}

In this section, we use a categorical description of Green products to define a similar construction in a more general context. The relevant property of $\Mon$ that we need for this generalization is that $\Mon$ is a \emph{pointed category}.

Let $\C$ be a category. An object of $\C$ which is both initial and terminal is called a \define{zero object}. If $\C$ has such an object, $\C$ is called a \define{pointed category} \cite{HomotAlg}. For any two objects $A, B$ of a pointed category, there is a unique map $0 \maps A \to B$ which is the composite of the unique map from $A$ to the zero object, and the unique map from the zero object to $B$.
If $\C$ is a pointed category with finite products, then for two objects $A, B$ of $\C$, the objects admit canonical maps $A \to A \times B$.
\[
\begin{tikzcd}
    &
    A
    \arrow[ddl, "1", bend right, swap]
    \arrow[d, "\exists!i_A"]
    \arrow[ddr, "0", bend left]
    \\&
    A \times B
    \arrow[dl, "\pi_A"]
    \arrow[dr, "\pi_B", swap]
    \\
    A
    &&
    B
\end{tikzcd}\]
So we have the following maps
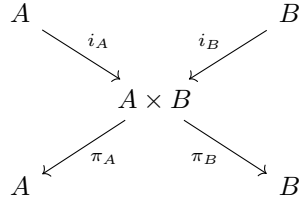
\[
\begin{tikzcd}
    A
    \arrow[dr, "i_A"]
    &&
    B
    \arrow[dl, "i_B", swap]
    \\&
    A \times B
    \arrow[dr, "\pi_B", swap]
    \arrow[dl, "\pi_A"]
    \\
    A
    &&
    B
\end{tikzcd}\]
satisfying the following properties.
\begin{align*}
    \pi_A i_A &= 1_A&
    \pi_B i_B &= 1_B\\
    \pi_B i_A &= 0&
    \pi_A i_B &= 0
\end{align*}
This is suggestive of a biproduct, but in a general pointed category $A \times B$ is not necessarily isomorphic to $A + B$.

In Section \ref{sec:funnetmod}, we use a generalized Green product to construct network models. 
A generalized Green product is a colimit of a diagram whose shape is derived from a given graph.
We describe the shapes of the diagrams here with \emph{quivers}.

We assume the reader is familiar with directed multi-graphs. We refer to them here as \emph{quivers} to help distinguish them from other variants of graphs and the role they play in this work. A \define {quiver} is a pair of sets $E$, $V$, respectively called the \emph{set of edges} and \emph{set of vertices}, and a pair of functions $s, t \maps E \to V$ assigning to each edge its \emph{starting} vertex and its \emph{terminating} vertex respectively. A \define{map of quivers} is a pair of functions 
\begin{displaymath}
\begin{tikzcd}
    E_1
    \arrow[d, shift right = 1, "s_1", swap]
    \arrow[d, shift left = 1, "t_1"]
    \arrow[r, "f_E"]
    &
    E_2
    \arrow[d, shift right = 1, "s_2", swap]
    \arrow[d, shift left = 1, "t_2"]
    \\
    V_1
    \arrow[r, "f_V", swap]
    &
    V_2
\end{tikzcd}
\end{displaymath}
such that the $s$-square and the $t$-square both commute.

We will use the word \define{cospan} to refer to the quiver with the following shape. \[\bullet \to \bullet \leftarrow \bullet\]
Define a functor $IC \maps \sGrph \to \Quiv$ which replaces every edge with a cospan ($IC$ stands for `insert cospan'). Specifically, given a simple graph $(V, E)$ where $E \subseteq \binom{V}{2}$, define the quiver $Q_1 \rightrightarrows Q_0$ where $Q_0 = V \sqcup E$ and $Q_1 = \{(v, e) \in V \times E |\, v \in e\}$, then define the source map $s \maps Q_1 \to Q_0$ by projection onto the first component, and the target map $t \maps Q_1 \to Q_0$ by projection onto the second component. For example, the simple graph 
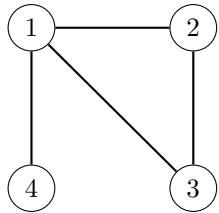
\[
\begin{tikzpicture}
    \node[style=species] (1) {$1$};
    \node[style=species] (2) [right = 1.5cm of 1]  {$2$};
    \node[style=species] (3) [below = 1.5cm and 1.5cm of 2] {$3$};
    \node[style=species] (4) [below = 1.5cm of 1] {$4$};

    \path[draw, thick]
    (1) edge node {} (2)
    (2) edge node {} (3)
    (3) edge node {} (1)
    (4) edge node {} (1);
\end{tikzpicture}\] 
gives the quiver 
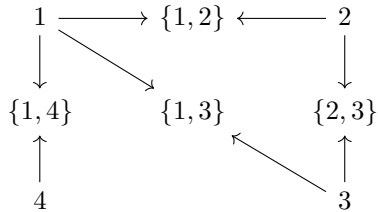
\[
\begin{tikzcd}
    1
    \arrow[r]
    \arrow[dr]
    \arrow[d]
    &
    \{1, 2\}
    &
    2
    \arrow[l]
    \arrow[d]
    \\
    \{1, 4\}
    &
    \{1, 3\}
    &
    \{2, 3\}
    \\
    4
    \arrow[u]
    &&
    3
    \arrow[u]
    \arrow[ul]
\end{tikzcd}\] 

Let $G=(V, E)$ and $G' = (V', E')$ be simple graphs, and $f \maps G \to G'$ a map of simple graphs. Define a map of quivers $ICf \maps IC(G) \to IC(G')$ by $ICf_0 = f_V \sqcup f_E$ and $ICf_1(v, e) = (f_V(v), f_E(e))$. 

\[
\begin{tikzcd}
    IC(G)_1
    \arrow[d, "s_G", bend right, swap]
    \arrow[d, "t_G", bend left]
    \arrow[r, "ICf_1"]
    &
    IC(G')_1
    \arrow[d, "s_{G'}", bend right, swap]
    \arrow[d, "t_{G'}", bend left]
    \\
    IC(G)_0
    \arrow[r, "IC f_0", swap]
    &
    IC(G')_0
\end{tikzcd}\]
This construction gives a coproduct preserving functor $IC \maps \sGrph \to \Quiv$.

Let $F \maps \Quiv \to \Cat$ denote the free category (or path category) functor \cite{CWM}. Since $F$ is a left adjoint, it preserves colimits. Notice that any quiver of the form $IC(G)$ would never have a path of length greater than 1. Thus the free path category on $IC(G)$ simply has identity morphisms adjoined. 

The objects in the category $F(IC(G))$ come from two places. There is an object for each vertex of $G$, and there is an object at the apex of the cospan for each edge in $G$. We call these two subsets of objects \define{vertex objects} and \define{edge objects}. We abuse notation and refer to the object given by the vertex $u$ by the same name, and similar for edge objects.

If $\{M_v\}_{v\in V}$ is a family of monoids indexed by the set $V$, that means that there is a functor $M \maps V \to \Mon$ from the set $V$ thought of as a discrete category. Notice that if $G$ is a simple graph with vertex set $V$, then the discrete category $V$ is a subcategory of $F(IC(G))$. 
We can then extend the functor $M$ to \[D \maps F(IC(G)) \to \Mon\] in the following way. Obviously we let $D(u) = M_u$ for a vertex object $u$. If $\{u, v\}$ is an edge in $G$, then $D(\{u, v\}) = M_u \times M_v$. The morphism $(u, \{u, v\})$ is sent to the canonical map $M_u \to M_u \times M_v$. For example, for a family of monoids $\{M_1, \dots M_4\}$, we have the following diagram.

\[
\begin{tikzcd}
    M_1
    \arrow[r]
    \arrow[dr]
    \arrow[d]
    &
    M_1 \times M_2
    &
    M_2
    \arrow[l]
    \arrow[d]
    \\
    M_1 \times M_4
    &
    M_1\times M_3
    &
    M_2\times M_3
    \\
    M_4
    \arrow[u]
    &&
    M_3
    \arrow[u]
    \arrow[ul]
\end{tikzcd}\] 
Since there are no non-trivial pairs of composable morphisms in categories of the form $F(IC(G))$, nothing further needs to be checked to confirm $D$ is a functor.

Despite the way we are denoting these products, we are not considering them to be ordered products. Alternatively, we could have used a more cumbersome notation that does not suggest any order on the factors.

\begin{thm}
    Let $V$ be a set, $\{M_v\}_{v \in V}$ be a family of monoids indexed by $V$, and $G = (V, E)$ be a simple graph with vertex set $V$. The $G$ Green product of $M_v$ is the colimit of the diagram $D \maps F(IC(G)) \to \Mon$ defined as above.
    \[
        G(M_v) \cong \colim D.
    \]
\end{thm}
\begin{proof}
    We show that $G(M_v)$ satisfies the necessary universal property. The vertex objects in the diagram have inclusion maps into the edge objects $i_{u, v} \maps M_u \to M_u \times M_v$, and all the objects have inclusion maps into $G(M_v)$, $j_u \maps M_u \to G(M_v)$ and $j_{u, v} \maps M_u \times M_v \to G(M_v)$ such that $j_{u, v} \circ i_{u, v} = j_u$. Note that due to the fact that we have unordered products for objects, there is some redundancy in our notation, namely $j_{u, v} = j_{v, u}$.
    If we have a monoid $Q$ and maps $f_u \maps M_u \to Q$ and $f_{u, v} \maps M_u \times M_v \to Q$ such that
    \begin{align*}
        f_{u, v} &= f_{v, u}\\
        f_{u, v} \circ i_{u, v} &= f_u, 
    \end{align*}
    then we define a map $\phi \maps G(M_v) \to Q$ by $\phi(m^{v_1}_1 \dots m^{v_k}_k) = f_{v_1}(m_1) \dots f_{v_k}(m_k)$. Since this map is defined via expressions of elements, Theorem \ref{thm:shuffle} tells us that to check this map is well-defined, we need only check that the values of two expressions that differ by a shuffle are the same. Let $m^{v_1}_1 \dots m^{v_k}_k$ be an expression, and $i$ such that $\{v_i, v_{i+1}\} \in E$. 
    \begin{align*}
        \phi (m^{v_i}_i m^{v_{i+1}}_{i+1})
        &= f_{v_i}(m_i) f_{v_{i+1}}(m_{i+1})
        \\&= f_{v_i, v_{i+1}} (m_i, m_{i+1})
        \\&= f_{v_{i+1}}(m_{i+1}) f_{v_i}(m_i)
        \\&= \phi (m^{v_{i+1}}_{i+1} m^{v_i}_i)
    \end{align*}
    It is clear that
    \[
        \phi (m^{v_1}_1 \dots m^{v_k}_k) = \phi (m^{v_1}_1 \dots m^{v_{i-1}}_{i-1}) \phi( m^{v_i}_im^{v_{i+1}}_{i+1}) \phi(m^{v_{i+2}}_{i+2} \dots m^{v_k}_k), 
    \] 
    so two shuffle equivalent expressions have the same value under $\phi$, and $\phi$ is well-defined. It is clearly a monoid homomorphism, and has the property $\phi \circ j_u = f_u$ and $\phi \circ j_{u, v} = f_{u, v}$. To show this map is unique, assume there is another such map $\psi \maps G(M_v) \to Q$. Since $\psi \circ j_u = f_u$, then $\psi(m_u) = f(u)$, and
    \begin{align*}
        \psi(m^{v_1}_1 \dots m^{v_k}_k) 
        &= \psi(m^{v_1}_1) \dots \psi(m^{v_k}_k)
        \\&= f_{v_1}(m_1) \dots f_{v_k}(m_k)
        \\&= \phi(m^{v_1}_1 \dots m^{v_k}_k).\qedhere
    \end{align*}
\end{proof}

This result makes it reasonable to generalize Green products in the following way.

\begin{defn}
    Let $\C$ be a pointed category with finite products and finite colimits, $V$ a set, $\{A_v\}_{v \in V}$ a family of objects of $\C$ indexed by $V$, and $G$ a simple graph with vertex set $V$. Let $D \maps F(IC(G)) \to \C$ be the diagram defined by $v \mapsto A_v$, $\{u, v\} \mapsto A_u \times A_v$, and the morphism $(u, \{u, v\})$ is mapped to the inclusion $A_u \to A_u \times A_v$ as above. The \define{$G$ Green product} of $\{A_v\}_{v \in V}$ is the colimit of $D$ in $\C$, 
    \[
        G^\C(A_v) = \colim D.
    \]
    If $\C=\Mon$, we denote the Green product simply as $G(A_v)$.
\end{defn}

In Section \ref{sec:funnetmod}, we use this general notion of graph products in \emph{varieties of monoids} to construct network models whose constituent monoids are in those varieties.
Note that since $F \circ IC$ is a functor, the group $\Aut(G)$ of graph automorphisms of $G$ naturally acts on $G^\C(A_v)$.

\subsection{Kneser Graphs}
\label{sec:kneser}

We focus here on a special family of simple graphs known as the \emph{Kneser graphs} \cite{Lovasz}. 
The \define{Kneser graph} $KG_{n,m}$ has vertex set $\binom{n}{m}$, the set of $m$-element subsets of an $n$-element set, and an edge between two vertices if they are disjoint subsets. Since a simple graph is defined as a collection of two-element subsets of an $n$-element set, the Kneser graph $KG_{n,2}$ has a vertex for each edge in the complete graph on $\n$, and has an edge between every pair of vertices which correspond to disjoint edges.
So the Kneser graph $\KG_{n,2}$ can be thought of as describing the disjointness of edges in the complete graph on $\n$. 
For instance, the complete graph on $\mathbf 4$ is 
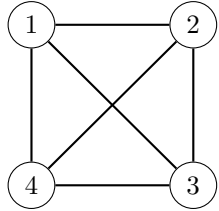
\[
\begin{tikzpicture}
    \node[style=species] (1) {$1$};
    \node[style=species] (2) [right = 1.5cm of 1]  {$2$};
    \node[style=species] (3) [below = 1.5cm and 1.5cm of 2] {$3$};
    \node[style=species] (4) [below = 1.5cm of 1] {$4$};

    \path[draw,thick]
    (1) edge node {} (2)
    (1) edge node {} (3)
    (1) edge node {} (4)
    (2) edge node {} (3)
    (2) edge node {} (4)
    (3) edge node {} (4);
\end{tikzpicture}\] 
and the corresponding Kneser graph $\KG_{4,2}$ is
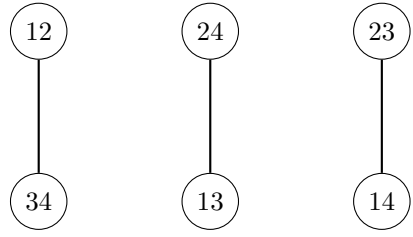
\[
\begin{tikzpicture}
    \node[style=species] (1) {$12$};
    \node[style=species] (2) [below = 1.5cm of 1] {$34$};
    \node[style=species] (3) [right = 1.5cm of 1] {$24$};
    \node[style=species] (4) [below = 1.5cm of 3] {$13$};
    \node[style=species] (5) [right = 1.5cm of 3] {$23$};
    \node[style=species] (6) [below = 1.5cm of 5] {$14$};

    \path[draw,thick]
    (1) edge node {} (2)
    (3) edge node {} (4)
    (5) edge node {} (6);
\end{tikzpicture}\] 
where we label the node with $uv$ if it corresponds to the edge $\{u,v\}$. One can see that an edge in the second graph corresponds exactly to a pair of disjoint edges in the first graph. For another example, $\KG_{5,2}$ is the Petersen graph. 
For sets $X,Y$ and a function $f \maps X \to Y$, let $f[U] = \{f(x) |\, x \in U\}$ for $U \subseteq X$. Let $\FinInj$ denote the category of finite sets and injective functions.

\begin{lem}
    For $k \in \N$, there is a functor $\bink{-} \maps \FinInj \to \FinInj$ which sends $X$ to $\bink{X}$ the set of $k$-element subsets of $X$, and injections $f \maps X \to Y$ to the functions $\bink{f} \maps \bink{X} \to \bink{Y}$ defined by $\bink{f}(U) = f[U]$. 
\end{lem}

Note that this result holds for $\Inj$ the category of sets and injective functions, but we only require $\FinInj$ for our purposes.

\begin{proof}
    If $f \maps X \to Y$ is an injection, then $|f[U]| = |U|$ for $U \subseteq X$. It then makes sense to restrict the induced map on power sets to subsets of a fixed cardinality. The map $\bink{f} \maps \bink{m} \to \bink{n}$ defined by $\bink{f}(U) = f[U]$ is then well defined. If $f[U]=f[V]$ and $x \in U$, then $f(x) \in f[U] = f[V]$, which implies there is a $y \in V$ such that $f(y) = f(x)$. Since $f$ is injective, then $x = y \in V$. Thus $U = V$ by symmetry.
\end{proof}

Let $i_X$ and $i_Y$ denote the following inclusion maps.
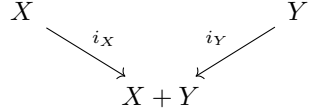
\[
\begin{tikzcd}
    X
    \arrow[dr,"i_X"]
    &&
    Y
    \arrow[dl,"i_Y",swap]
    \\&
    X+Y
\end{tikzcd}\]
Since these maps are injective, they induce maps $\bink{i_X}, \bink{i_Y}$, and we get a map $\Phi_{X,Y} \maps \bink{X}+\bink{Y} \to \bink{X+Y}$ by the universal property in the following way.
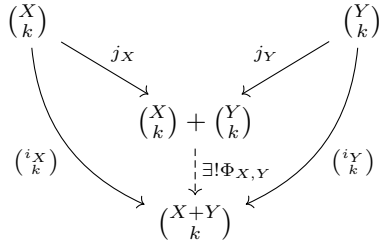
\[
\begin{tikzcd}
    \bink{X}
    \arrow[dr,"j_X"]
    \arrow[ddr,"\bink{i_X}",bend right,swap]
    &&
    \bink{Y}
    \arrow[dl,"j_Y",swap]
    \arrow[ddl,"\bink{i_Y}",bend left]
    \\&
    \bink{X}+\bink{Y}
    \arrow[d,"\exists!\Phi_{X,Y}",dashed]
    \\&
    \bink{X+Y}
\end{tikzcd}\]

\begin{lem}
\label{lem:chooselax}
    The functor $\bink{-}$ is made lax symmetric monoidal \[(\bink{-},\Phi, \phi) \maps (\FinInj,+, \emptyset) \to (\FinInj, +, \emptyset)\] where the components of $\Phi$ are defined as above.
\end{lem}
\begin{proof}
    The family of maps $\{\Phi_{X,Y}\}$ is clearly a natural transformation.
    There is no choice for the map $\phi \maps \emptyset \to \bink{\emptyset}$.
    The left and right unitor laws hold trivially. Checking the coherence conditions for the associator and the symmetry are straightforward computations.
\end{proof}

For $n,k \in \N$, the simple graph $KG_{n,k}$ has vertex set $V = \bink{\n}$ and edge set $\{\{u,v\} \subseteq \binom{V}{2} |\, u \cap v = \emptyset\}$. 
If $f \maps \m \to \n$ is injective, then we get a map $\bink{f}$ between the vertex sets of $KG_{m,k}$ and $KG_{n,k}$. Let $\{u,v\} \in \binom{V}{2}$ be an edge in $KG_{m,k}$. Then $f[u] \cap f[v] = \emptyset$ by injectivity, so $\{f[u],f[v]\}$ is an edge of $KG_{n,k}$. An injection $f$ then induces a map of graphs, denoted $KG_{f,k} \maps KG_{m,k} \to KG_{n,k}$. 
Since $\bink{f}$ is injective, $KG_{f,k}$ is an embedding.
Nothing about this construction requires finiteness of the sets involved, but our applications only call for finite graphs. 

\begin{prop}
    For $k \in \N$, there is functor $KG_{-,k} \maps \FinInj \to \sGrph$ which sends $\n$ to $KG_{n,k}$ and $f\maps \m \to \n$ to $KG_{f,k}$.
\end{prop}

Not only does $KG_{m,k}$ embed into $KG_{n,k}$ when $m<n$, but $KG_{m,k} + KG_{n,k}$ embeds into $KG_{m+n,k}$. 
We construct the embedding $KG_{m,k} + KG_{n,k} \to KG_{m+n,k}$ by using the lax structure map from Lemma \ref{lem:chooselax} for the vertex map, $\Phi_{\m,\n} \maps \bink{\m} + \bink{\n} \to \bink{\m+\n}$. Restricting this map to either $\bink{\m}$ (resp. $\bink{\n}$) gives the map $\bink{i_\m}$ (resp. $\bink{i_\n}$) which we already know induces a map of graphs. Thus $\Phi_{\m,\n}$ induces a map of graphs, which we call $\Psi_{\m,\n}$.

\begin{prop}
    The functor $KG_{-,k}$ is made lax (symmetric) monoidal \[(KG_{-,k},\Psi) \maps (\Inj,+) \to (\sGrph, +)\] where the components of $\Psi$ are defined as above.
\end{prop}
\begin{proof}
    All the necessary properties for $\Psi$ are inherited immediately from $\Phi$.
\end{proof}

Let $(L, \Lambda) \maps (\Inj,+) \to (\Cat,+)$ be the composite $L = F \circ IC \circ KG_{-,2}$ with the obvious laxator. Let $M$ be a monoid. Then from the construction given in the previous subsection, for each $\n$ we get a diagram $D_n \maps L(\n) \to \Mon$ which sends all vertex objects to $M$, all edge objects to $M \times M$, and all nontrivial morphisms to inclusions $M \to M \times M$. 
Taking the colimit of $D_n$ then gives the Green product $\KG_{n,2}(M)$. 

Note that we identify constituent monoids with the corresponding submonoid of the graph product when this can be done without confusion.

\begin{prop}
\label{prop:knesercomm}
    Let $M_{p,q}$ be a $\binom{m+n}{2}$ family of monoids, and $G_1$ and $G_2$ be graphs with $m$ and $n$ vertices respectively. Let $a_1 \in M_{p_1,q_1}$ with $p_1,q_1 \leq m$ and $a_2 \in M_{p_2,q_2}$ with $p_2,q_2>m$, and let $\overline a_1, \overline a_2$ be their values under the canonical inclusions $M_{p,q} \hookrightarrow (G_1 \sqcup G_2)(M_{p,q})$. Then $\overline a_1 \overline a_2 = \overline a_2 \overline a_1$ in $(G_1 \sqcup G_2)(M_{p,q})$.
\end{prop}

\begin{proof}
    By assumption, there is an edge in $KG_{m+n,2}$ between the vertices ${p_1,q_1}$ and ${p_2,q_2}$.
\end{proof}

\subsection{Varieties of Monoids}
\label{sec:varmon}

A \define{finitary algebraic theory} or \define{Lawvere theory} is a category $T$ with finite products in which every object is isomorphic to a finite cartesian power $x^n = \prod^n x$ of a distinguished object $x$ \cite{Lawvere, ALR}. 
An \define{algebra} of a theory $T$, or \define{$T$-algebra}, is a product preserving functor $T \to \Set$. 
Let $T\Alg$ denote the category of $T$-algebras with natural transformations for morphisms. 
We are primarily concerned with monoids in this paper. 
The theory of monoids $T_\Mon$ has morphisms $m \maps x\times x \to x$ and $e \maps x^0 \to x$, which makes the following diagrams commute.
\[
\begin{tikzcd}
    x^3
    \arrow[d, "m \times 1_x", swap]
    \arrow[r, "1_x \times m"]
    &
    x^2
    \arrow[d, "m"]
    &
    x \times x^0
    \arrow[r, "1_x \times e"]
    \arrow[dr, "\simeq", swap]
    &
    x^2
    \arrow[d, "m"]
    &
    x^0 \times x
    \arrow[l, "e \times 1_x", swap]
    \arrow[dl, "\simeq"]
    \\
    x^2
    \arrow[r, "m", swap]
    &
    x
    &&
    x
\end{tikzcd}\]

A \define{variety} of $T$-algebras is a full subcategory of $T\Alg$ which is closed under products, subobjects, and homomorphic images. Birkhoff's theorem implies that this is equivalent to the category $T'\Alg$ of algebras of another theory $T'$ which has the same morphisms, but satisfies more commutative diagrams \cite{universal}.
For example, commutative monoids are given by algebras of the theory of commutative monoids $T_\CMon$, which has morphisms $m,e$ as in $T_\Mon$, satisfies the same commutative diagrams as $T_\Mon$, but also satisfies the following commutative diagram
\[
\begin{tikzcd}
    x^2
    \arrow[dr, "m", swap]
    \arrow[r, "b"]
    &
    x^2
    \arrow[d, "m"]
    \\&
    x
\end{tikzcd}\]
where $b: x^2 \to x^2$ is the braid isomorphism. We only use varieties of monoids in this paper, so we give these ``extra'' conditions by equations, e.g.\ commutative monoids are those which satisfy the equation $ab=ba$ for all elements $a,b$. We call the extra equations the \emph{defining equations} of the variety.

A \define{graphic monoid} is a monoid which satisfies the \emph{graphic identity}: $aba=ab$ for all elements $a,b$. Graphic monoids are algebras of a theory $T_\GMon$. A semigroup obeying this relation is known as a \emph{left regular band} \cite{MSS}. The term \emph{graphic monoid} was introduced by Lawvere \cite{taco}. Let $M$ be a graphic monoid. If we let $b$ be the unit of $M$, then the graphic relation says that $a^2=a$. Every element of $M$ is idempotent. If $a,c \in M$, then $ca=c$ if $c$ already has $a$ as a factor. 

Graphic monoids are present when talking about types of information where a piece of information cannot contain the same piece of information twice. A simple example can be seen in the powerset of a given set $X$, given the structure of a monoid by union. Of course, this example is overly simple because the operation is commutative idempotent, which is stronger than graphic. A more interesting example can be seen by considering the following simple graph.
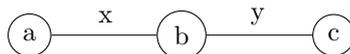
\[
\begin{tikzpicture}
	\begin{pgfonlayer}{nodelayer}
		\node [style=construct] (1) at (-2, 0) {a};
		\node [style=construct] (2) at (0, 0) {b};
		\node [style=construct] (3) at (2, 0) {c};
		\node [style=none] (4) at (-1, 0.25) {x};
		\node [style=none] (5) at (1, 0.25) {y};
	\end{pgfonlayer}
	\begin{pgfonlayer}{edgelayer}
		\draw (1) to (2);
		\draw (2) to (3);
	\end{pgfonlayer}
\end{tikzpicture}\]
We will define a monoid structure on the set $M = \{1, a, b, c, x, y\}$ in the following way. First, $1$ is a freely adjoined identity element. For $p,q \in M \setminus \{1\}$, define $pq$ as follows. Pick a generic point $f$ in $p$ and a generic point $g$ in $q$. Then move a small distance along a straight line path from $f$ to $g$. We define the product $pq$ to be the component of the graph you land in. Here are some example computations:
\begin{align*}
    &ab = x
    &aa = a
    \\
    &bc = y
    &xb = x
    \\
    &ac = x
    &ca = y
\end{align*}
The last two demonstrate that this monoid is not commutative. More complicated examples can be constructed by using the same idea for the operation, but applying it to different spaces.

Our motivation for using graphic monoids is that we use the graphic relation to model "commitment" in the following way. Let $M$ be a graphic monoid, where we think of an element of $M$ as a task or list of tasks. If we first commit to doing task $x$, and then commit to doing task $y$, then we have the element $xy$ as our task list, indicating that we committed to $x$ before $y$. If we then try to commit to to doing $x$, the graphic relation saves us from recording this information twice. The relation also preserves the order in which we committed to $x$ and $y$: if $x$ is a task list of the form $x = ab$, and we have committed to $xy$, and then try to commit to $bc$, we get $(xy)(bc) = (aby)(bc) = a (byb) c = a (by) c = abyc = xyc$.

We want to construct a network model from a monoid in a variety $\V$ which has constituent monoids that are also in $\V$. 
If $M$ is a monoid in a variety $\V$, then each constituent monoid $\G_M(n)$ is a product of several copies of $M$, and so is also in $\V$ by definition. 
Thus the ordinary network model (given in Theorem \ref{thm:ordinary}) restricted to a variety gives a functor $\V \to \NM_\V$, where $\NM_\V$ denotes the category of $\V$-valued network models.

The free product of two monoids is a monoid, $M+N$ an element of which is given by a list with entries in the set $M \sqcup N$ such that if two consecutive entries of a list are either both elements of $M$ or both elements of $N$, then the list is identified with the list that is the same everywhere except that those two entries are reduced to one entry occupied by their product. Note that the empty list is identified with both the singleton list consisting of the identity element of $M$, and the singleton list consisting of the identity element of $N$. Free products of monoids gives the coproduct in the category of monoids $\Mon$. Free products of monoids are very similar to free products of groups, which can be found in most books introducing group theory \cite{Hungerford}. 

If two monoids $M$ and $N$ are in a variety $\V$, taking their free product will not necessarily produce a monoid in $\V$, i.e.\ varieties are not necessarily closed under the coproduct of $\Mon$. It is easy to find an example demonstrating this. 
Consider $\IMon$, the variety of idempotent monoids, i.e.\ monoids satisfying the equation $x^2 = x$ for all elements $x$. The boolean monoid $\B$ is an object in $\IMon$. The free product of $\B$ with itself $\B+\B$ can be generated by elements $a$ and $b$ which correspond to the element $1$ in each copy of $\B$. The element $ab \in \B+\B$ is not idempotent, as $abab \neq ab$. However, every variety $\V$ does have coproducts. The coproduct in a variety of monoids is the quotient of the free product by the congruence relation generated by the variety's defining equations. In Section \ref{sec:funnetmod} we give a construction $\V \to \NM_\V$ which uses colimits in order to impose minimal relations.

We need the following fact for the main construction of this paper. It follows immediately from the definitions.
\begin{lem}
    Every variety of monoids is a pointed category and has finite colimits.
\end{lem}

This lemma tells us that it makes sense to talk about Green products in a variety, which we call \emph{varietal Green products}. In the next section, we use varietal Green products with Kneser graphs to construct network models.

\section{Functorial Network Models}
\label{sec:funnetmod}

In this section, we state and prove the main result of this paper.
It says that given a monoid $M$ in a variety $\V$, we can construct a network model whose constituent monoids are also in $\V$, while avoiding to impose commutativity relations when possible. 
In the following section, we see how this construction resolves the dilemma presented in Question \ref{question}.

Let $M$ be a monoid in a variety $\V$. Define $\GMV(n)$ to be the $KG_{n,2}$ Green product of $\binom{n}{2}$ copies of $M$.

\begin{thm} 
\label{thm:main}
    For $\V$ a variety of monoids, $\G_{-,\V} \maps \V \to \NM_\V$ is a functor, as given above. The network model $\GMV$ is called the \define{$\V$-varietal network model for $M$-weighted graphs}, or just the \define{varietal $M$ network model}.
\end{thm}

In order to prove this, we must first show that a monoid $M$ gives a network model, i.e.\ a lax symmetric monoidal functor. The laxator for $\GMV$ is canonically defined, but perhaps it is not as immediate as the one for the ordinary $M$ network model. We treat this first before returning to the proof of the main theorem.

Let $A$ and $B$ be objects in a pointed category with finite products and coproducts. Let $p_A \maps A \times B \to A$ and $p_B \maps A \times B \to B$ denote the canonical projections, and $i_A \maps A \to A+B$ and $i_B \maps B \to A + B$ the canonical inclusions.
The category $\CMon$ of commutative monoids is such a category. Recall that the operation of a monoid is a monoid homomorphism if and only if the monoid is commutative.
We have
\[
\begin{tikzcd}
    &
    A \times B
    \arrow[dl, "p_A", swap]
    \arrow[dr, "p_B"]
    \arrow[d, dashed]
    \\
    A
    \arrow[d, "i_A", swap]
    &
    (A + B) \times (A + B)
    \arrow[d,"\ast"]
    \arrow[dl]
    \arrow[dr]
    &
    B
    \arrow[d, "i_B"]
    \\
    A + B
    &
    A + B
    &
    A + B
\end{tikzcd}
\]
where $\ast$ denotes the operation in the commutative monoid $A+B$, and the dashed arrow is $<i_A p_A, i_B p_B>$ given by universal property.
The composite of the two maps going down the middle is the inverse to the canonical map $A + B \to A \times B$.
The operation in a noncommutative monoid is not a monoid homomorphism, but all the above maps still exist \emph{as functions}.
Recall that we let $\cup$ denote the operation in the monoids $\GMV(\n)$. There is always a homomorphism $\phi_{\m,\n} \maps \GMV(\m) + \GMV(\n) \to \GMV(\m+\n)$ by universal property of coproducts. 
Let \[\gamma \maps (\GMV(\m) + \GMV(\n)) \times (\GMV(\m) + \GMV(\n)) \to \GMV(\m) + \GMV(\n)\] denote the monoid operation of the coproduct.

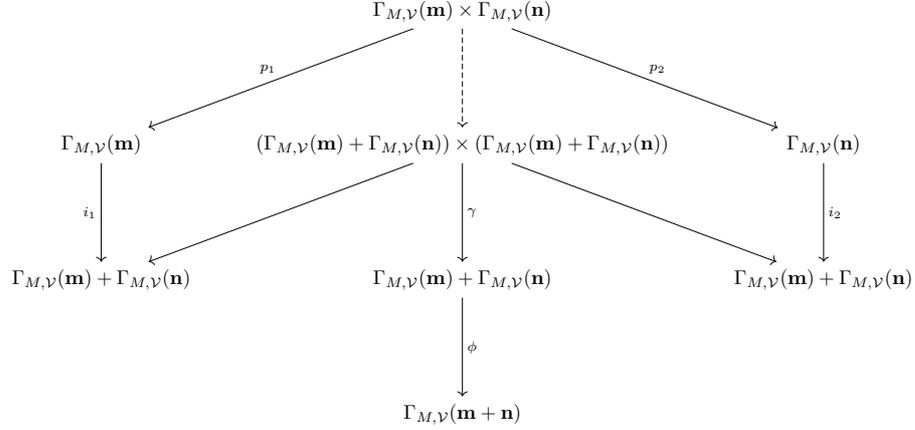
\adjustbox{scale = 0.75, center}{
\begin{math}\begin{tikzcd}[row sep = 50]
    &
    \GMV(\m) \times \GMV(\n)
    \arrow[dl, "p_1", swap]
    \arrow[dr, "p_2"]
    \arrow[d, dashed]
    \\
    \GMV(\m)
    \arrow[d, "i_1", swap]
    &
    (\GMV(\m)+\GMV(\n))\times(\GMV(\m)+\GMV(\n))
    \arrow[d,"\gamma"]
    \arrow[dl]
    \arrow[dr]
    &
    \GMV(\n)
    \arrow[d, "i_2"]
    \\
    \GMV(\m)+\GMV(\n)
    &
    \GMV(\m)+\GMV(\n)
    \arrow[d, "\phi"]
    &
    \GMV(\m)+\GMV(\n)
    \\&
    \GMV(\m+\n)
\end{tikzcd}\end{math}
}
The monoids $\GMV(n)$ are constructed specifically so that $\phi \circ \gamma \circ <i_1 \circ p_1, i_2 \circ p_2>$ is a monoid homomorphism despite the fact that $\gamma$ is not.

In the proof of the following theorem, we utilize a string diagrammatic calculus suited for reasoning in a symmetric monoidal category. We refer the reader to Selinger's thorough exposition of such string diagramatic languages and their use in category theory \cite{Selinger}.

\begin{lem}
    The function $\GMV(\m) \times \GMV(\n) \to \GMV(\m + \n)$ given by $\phi \circ (i_1 \circ p_1 \cup i_2 \circ p_2)$ is a monoid homomorphism. Moreover, the family of maps of this form gives a natural transformation, denoted $\sqcup$.
\end{lem}

\begin{proof}
    We have the following actors in play:
    \begin{itemize}
        \item the monoid operations $\cup_k \maps \GMV(\mathbf k)$ for $k = m, n, m+n$ (we leave off the subscripts below)
        \item the monoid operation of the coproduct \[\gamma \maps (\GMV(\m) + \GMV(\n)) \times (\GMV(\m) + \GMV(\n)) \to \GMV(\m) + \GMV(\n)\]
        \item the canonical inclusion maps $i_1 \maps \GMV(\m) \to \GMV(\m) \to \GMV(\n)$ and $i_2 \maps \GMV(\n) \to \GMV(\m) \to \GMV(\n)$
        \item the canonical map $\phi \maps \GMV(\m) + \GMV(\n) \to \GMV(\m+\n)$
    \end{itemize}
    We represent these string diagramatically (read from top to bottom) as follows. Note that these are digrams in $\Set$ with its cartesian monoidal structure, because the monoid operations $\cup_k$ and $\gamma$ are not necessarily monoid homomorphisms.
    \[
    \begin{tikzpicture}
    \begin{pgfonlayer}{nodelayer}
    	\node [style=construct] (U) at (0, 0) {$\cup$};
        \node [style=none] (1) at (-0.5, 1) {};
    	\node [style=none] (2) at (-0.5, 0.5) {};
    	\node [style=none] (3) at (0.5, 0.5) {};
    	\node [style=none] (4) at (0, -1) {};
    	\node [style=none] (5) at (0.5, 1) {};
        \node [style=none] () at (0.75, 0) {,};
        \node [style=none] () at (1, 0) {}; 
    \end{pgfonlayer}
    \begin{pgfonlayer}{edgelayer}
    	\draw [bend right] (2.center) to (U);
    	\draw [bend left] (3.center) to (U);
    	\draw [] (2.center) to (1);
    	\draw [] (3.center) to (5);
    	\draw (4.center) to (U);
    \end{pgfonlayer}
    \end{tikzpicture}
    \begin{tikzpicture}
    \begin{pgfonlayer}{nodelayer}
    	\node [style=construct] (g) at (0, 0) {$\gamma$};
        \node [style=none] (1) at (-0.5, 1) {};
    	\node [style=none] (2) at (-0.5, 0.5) {};
    	\node [style=none] (3) at (0.5, 0.5) {};
    	\node [style=none] (4) at (0, -1) {};
    	\node [style=none] (5) at (0.5, 1) {};
        \node [style=none] () at (0.75, 0) {,};
        \node [style=none] () at (1, 0) {}; 
    \end{pgfonlayer}
    \begin{pgfonlayer}{edgelayer}
    	\draw [bend right] (2.center) to (g);
    	\draw [bend left] (3.center) to (g);
    	\draw [] (2.center) to (1);
    	\draw [] (3.center) to (5);
    	\draw (4.center) to (g);
    \end{pgfonlayer}
    \end{tikzpicture}
    \begin{tikzpicture}
    \begin{pgfonlayer}{nodelayer}
    	\node [style=construct] (1) at (0, 0) {$i_1$};
    	\node [style=none] (2) at (0, 1) {};
    	\node [style=none] (3) at (0, -1) {};
        \node [style=none] (5) at (0.75, 0) {,};
        \node [style=none] (5) at (1, 0) {}; 
    \end{pgfonlayer}
    \begin{pgfonlayer}{edgelayer}
    	\draw (2.center) to (1);
    	\draw (3.center) to (1);
    \end{pgfonlayer}
    \end{tikzpicture}
    \begin{tikzpicture}
    \begin{pgfonlayer}{nodelayer}
    	\node [style=construct] (1) at (0, 0) {$i_2$};
    	\node [style=none] (2) at (0, 1) {};
    	\node [style=none] (3) at (0, -1) {};
        \node [style=none] (5) at (0.75, 0) {,};
        \node [style=none] (5) at (1, 0) {}; 
    \end{pgfonlayer}
    \begin{pgfonlayer}{edgelayer}
    	\draw (2.center) to (1);
    	\draw (3.center) to (1);
    \end{pgfonlayer}
    \end{tikzpicture}
    \begin{tikzpicture}
    \begin{pgfonlayer}{nodelayer}
    	\node [style=construct] (1) at (0, 0) {$\phi$};
    	\node [style=none] (2) at (0, 1) {};
    	\node [style=none] (3) at (0, -1) {};
    \end{pgfonlayer}
    \begin{pgfonlayer}{edgelayer}
    	\draw (2.center) to (1);
    	\draw (3.center) to (1);
    \end{pgfonlayer}
    \end{tikzpicture}
    \]
    We define $\sqcup \maps \GMV(\m) \times \GMV(\n) \to \GMV(\m + \n)$ as follows.
    \begin{equation}\label{def:disjoint}
    \begin{tikzpicture}[baseline=(current bounding  box.center)]
    \begin{pgfonlayer}{nodelayer}
    	\node [style=construct] (U) at (0, 0) {$\sqcup$};
    	\node [style=none] (1) at (-0.5, 1) {};
    	\node [style=none] (2) at (0.5, 1) {};
    	\node [style=none] (3) at (-0.5, 0.5) {};
    	\node [style=none] (4) at (0.5, 0.5) {};
    	\node [style=none] (5) at (0, -1) {};
        \node [style=none] () at (1.2, 0) {=};
        \node [style=none] () at (1.8, 0) {}; 
        \node [style=none] () at (0, -1.6) {}; 
    \end{pgfonlayer}
    \begin{pgfonlayer}{edgelayer}
    	\draw [bend right] (3.center) to (U);
    	\draw [bend left] (4.center) to (U);
    	\draw [] (3.center) to (1);
    	\draw [] (4.center) to (2);
    	\draw (5.center) to (U);
    \end{pgfonlayer}
    \end{tikzpicture}
    \begin{tikzpicture}[baseline=(current bounding  box.center)]
    \begin{pgfonlayer}{nodelayer}
    	\node [style=none] (4) at (0, -1) {};
    	\node [style=construct] (6) at (-0.5, 1) {$\phi$};
    	\node [style=construct] (7) at (0.5, 1) {$\phi$};
    	\node [style=construct] (15) at (0, 0) {$\cup$};
    	\node [style=none] (16) at (-0.5, 3) {};
    	\node [style=none] (17) at (0.5, 3) {};
    	\node [style=construct] (18) at (-0.5, 2) {$i_1$};
    	\node [style=construct] (19) at (0.5, 2) {$i_2$};
    \end{pgfonlayer}
    \begin{pgfonlayer}{edgelayer}
    	\draw (15) to (4.center);
    	\draw [bend right] (6) to (15);
    	\draw [bend left] (7) to (15);
    	\draw (18) to (6);
    	\draw (19) to (7);
    	\draw (16.center) to (18);
    	\draw (17.center) to (19);
    \end{pgfonlayer}
    \end{tikzpicture}
    \end{equation}
    Proposition \ref{prop:knesercomm} gives the following equation.
    \begin{equation}\label{prop13}
    \begin{tikzpicture}[baseline=(current  bounding  box.center)]
    \begin{pgfonlayer}{nodelayer}
    	\node [style=none] (4) at (0, -1) {};
    	\node [style=construct] (6) at (-0.5, 1) {$\phi$};
    	\node [style=construct] (7) at (0.5, 1) {$\phi$};
    	\node [style=construct] (15) at (0, 0) {$\cup$};
    	\node [style=none] (16) at (-0.5, 3) {};
    	\node [style=none] (17) at (0.5, 3) {};
    	\node [style=construct] (18) at (-0.5, 2) {$i_2$};
    	\node [style=construct] (19) at (0.5, 2) {$i_1$};
    	\node [style=none] (1) at (1.8, 1) {=};
    \end{pgfonlayer}
    \begin{pgfonlayer}{edgelayer}
    	\draw (15) to (4.center);
    	\draw [bend right] (6) to (15);
    	\draw [bend left] (7) to (15);
    	\draw (18) to (6);
    	\draw (19) to (7);
    	\draw (16.center) to (18);
    	\draw (17.center) to (19);
    \end{pgfonlayer}
    \end{tikzpicture}
    \begin{tikzpicture}[baseline=(current  bounding  box.center)]
    \begin{pgfonlayer}{nodelayer}
        \node [style=none] () at (-1.8, 1) {}; 
    	\node [style=construct] (f1) at (-0.5, 1) {$\phi$};
    	\node [style=construct] (f2) at (0.5, 1) {$\phi$};
    	\node [style=construct] (U) at (0, 0) {$\cup$};
    	\node [style=construct] (i1) at (-0.5, 2) {$i_1$};
    	\node [style=construct] (i2) at (0.5, 2) {$i_2$};
    	\node [style=none] (1) at (0, -1) {};
    	\node [style=none] (2) at (-0.5, 2.5) {};
    	\node [style=none] (3) at (0.5, 2.5) {};
    	\node [style=none] (4) at (0, 3) {};
    	\node [style=none] (5) at (-0.5, 3.5) {};
    	\node [style=none] (6) at (0.5, 3.5) {};
    \end{pgfonlayer}
    \begin{pgfonlayer}{edgelayer}
    	\draw [bend right] (5) to (4.center);
    	\draw [bend left] (6) to (4.center);
    	\draw [bend left] (4.center) to (3.center);
    	\draw [bend right] (4.center) to (2.center);
    	\draw (2.center) to (i1);
    	\draw (3.center) to (i2);
    	\draw (i1) to (f1);
    	\draw (i2) to (f2);
    	\draw [bend right] (f1) to (U);
    	\draw [bend left] (f2) to (U);
    	\draw (U) to (1);
    \end{pgfonlayer}
    \end{tikzpicture}
    \end{equation}
    Since $\phi$ is a homomorphism, we get the following equation.
    \begin{equation}\label{phiishom}
    \begin{tikzpicture}[baseline=(current  bounding  box.center)]
    \begin{pgfonlayer}{nodelayer}
    	\node [style=construct] (1) at (0, 0) {$\cup$};
    	\node [style=construct] (2) at (-0.5, 1) {$\phi$};
    	\node [style=construct] (3) at (0.5, 1) {$\phi$};
    	\node [style=none] (4) at (-0.5, 2) {};
    	\node [style=none] (5) at (0.5, 2) {};
    	\node [style=none] (6) at (0, -1) {};
    	\node [style=none] (7) at (1.8, 0.5) {=};
    	\node [style=none] (8) at (2.8, 0) {};
    \end{pgfonlayer}
    \begin{pgfonlayer}{edgelayer}
    	\draw (4) to (2);
    	\draw (5) to (3);
    	\draw [bend right] (2) to (1);
    	\draw [bend left] (3) to (1);
    	\draw (1) to (6);
    \end{pgfonlayer}
    \end{tikzpicture}
    \begin{tikzpicture}[baseline=(current  bounding  box.center)]
    \begin{pgfonlayer}{nodelayer}
    	\node [style=construct] (g) at (0, 1) {$\gamma$};
    	\node [style=construct] (f) at (0, 0) {$\phi$};
    	\node [style=none] (1) at (-0.5, 2) {};
    	\node [style=none] (2) at (0.5, 2) {};
    	\node [style=none] (3) at (-0.5, 1.5) {};
    	\node [style=none] (4) at (0.5, 1.5) {};
    	\node [style=none] (5) at (0, -1) {};
    \end{pgfonlayer}
    \begin{pgfonlayer}{edgelayer}
        \draw (1) to (3.center);
        \draw (2) to (4.center);
    	\draw [bend right] (3.center) to (g);
    	\draw [bend left] (4.center) to (g);
    	\draw (g) to (f);
    	\draw (f) to (5);
    \end{pgfonlayer}
    \end{tikzpicture}
    \end{equation}
    Since $i_1$ and $i_2$ are homomorphisms, we get the following equations.
    \begin{equation}\label{isarehoms}
    \begin{tikzpicture}[baseline=(current  bounding  box.center)]
    \begin{pgfonlayer}{nodelayer}
    	\node [style=construct] (1) at (0, 0) {$\gamma$};
    	\node [style=construct] (2) at (-0.5, 1) {$i_j$};
    	\node [style=construct] (3) at (0.5, 1) {$i_j$};
    	\node [style=none] (4) at (-0.5, 2) {};
    	\node [style=none] (5) at (0.5, 2) {};
    	\node [style=none] (6) at (0, -1) {};
    	\node [style=none] (7) at (1.8, 0.5) {=};
    	\node [style=none] (8) at (2.8, 0) {};
    \end{pgfonlayer}
    \begin{pgfonlayer}{edgelayer}
    	\draw (4) to (2);
    	\draw (5) to (3);
    	\draw [bend right] (2) to (1);
    	\draw [bend left] (3) to (1);
    	\draw (1) to (6);
    \end{pgfonlayer}
    \end{tikzpicture}
    \begin{tikzpicture}[baseline=(current  bounding  box.center)]
    \begin{pgfonlayer}{nodelayer}
    	\node [style=construct] (g) at (0, 1) {$\gamma$};
    	\node [style=construct] (f) at (0, 0) {$i_j$};
    	\node [style=none] (1) at (-0.5, 2) {};
    	\node [style=none] (2) at (0.5, 2) {};
    	\node [style=none] (3) at (-0.5, 1.5) {};
    	\node [style=none] (4) at (0.5, 1.5) {};
    	\node [style=none] (5) at (0, -1) {};
    \end{pgfonlayer}
    \begin{pgfonlayer}{edgelayer}
        \draw (1) to (3.center);
        \draw (2) to (4.center);
    	\draw [bend right] (3.center) to (g);
    	\draw [bend left] (4.center) to (g);
    	\draw (g) to (f);
    	\draw (f) to (5);
    \end{pgfonlayer}
    \end{tikzpicture}
    \end{equation}
    
    We want to show that $(g \sqcup h) \cup (g' \sqcup h') = (g \cup g') \sqcup (h \cup h')$. We compute:
    
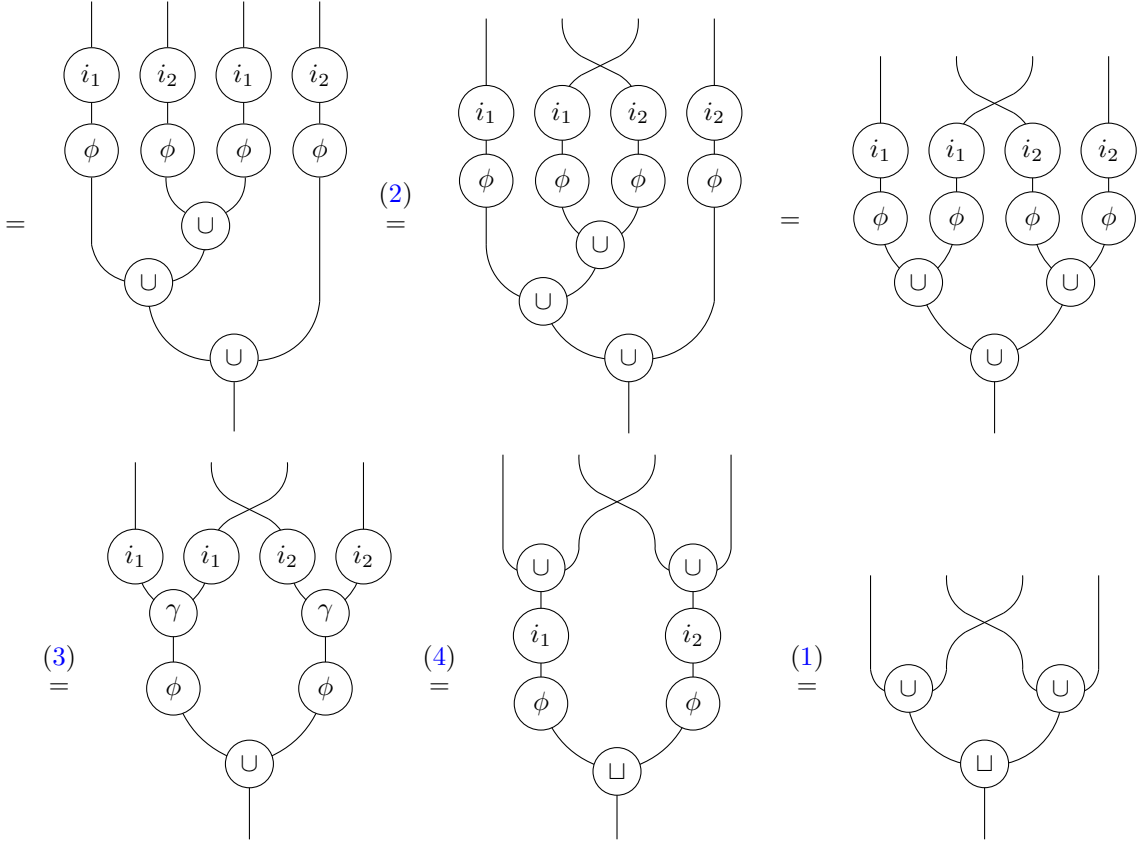
\[
    \begin{tikzpicture}
    \begin{pgfonlayer}{nodelayer}
    	\node [style=construct] (s1) at (-0.75, 1) {$\sqcup$};
    	\node [style=construct] (s2) at (0.75, 1) {$\sqcup$};
    	\node [style=construct] (U) at (0, 0) {$\cup$};
    	\node [style=none] (1) at (-1.25, 1.5) {};
    	\node [style=none] (2) at (0.25, 1.5) {};
    	\node [style=none] (3) at (-0.25, 1.5) {};
    	\node [style=none] (4) at (1.25, 1.5) {};
    	\node [style=none] (5) at (0, -1) {};
    	\node [style=none] (6) at (-1.25, 2) {};
    	\node [style=none] (7) at (0.25, 2) {};
    	\node [style=none] (8) at (-0.25, 2) {};
    	\node [style=none] (9) at (1.25, 2) {};
    	\node [style=none] () at (2.2, 0.5) {=};
    	\node [style=none] () at (2.2, 0.9) {(\ref{def:disjoint})};
    	\node [style=none] () at (2.8, 1) {}; 
    \end{pgfonlayer}
    \begin{pgfonlayer}{edgelayer}
        \draw (6) to (1.center);
        \draw (8) to (3.center);
        \draw (7) to (2.center);
        \draw (9) to (4.center);
        \draw [bend right] (1.center) to (s1);
        \draw [bend left] (3.center) to (s1);
        \draw [bend right] (2.center) to (s2);
        \draw [bend left] (4.center) to (s2);
        \draw [bend right] (s1) to (U);
        \draw [bend left] (s2) to (U);
        \draw (U) to (5);
    \end{pgfonlayer}
    \end{tikzpicture}
    \begin{tikzpicture}
    \begin{pgfonlayer}{nodelayer}
    	\node [style=construct] (U1) at (-1, -1) {$\cup$};
    	\node [style=construct] (U2) at (1, -1) {$\cup$};
    	\node [style=construct] (U3) at (0, -2) {$\cup$};
    	\node [style=construct] (i1) at (-1.5, 1) {$i_1$};
    	\node [style=construct] (i2) at (-0.5, 1) {$i_2$};
    	\node [style=construct] (i1') at (0.5, 1) {$i_1$};
    	\node [style=construct] (i2') at (1.5, 1) {$i_2$};
    	\node [style=construct] (f1) at (-1.5, 0) {$\phi$};
    	\node [style=construct] (f2) at (-0.5, 0) {$\phi$};
    	\node [style=construct] (f3) at (0.5, 0) {$\phi$};
    	\node [style=construct] (f4) at (1.5, 0) {$\phi$};
    	\node [style=none] (1) at (-1.5, 2) {};
    	\node [style=none] (2) at (0.5, 2) {};
    	\node [style=none] (3) at (-0.5, 2) {};
    	\node [style=none] (4) at (1.5, 2) {};
    	\node [style=none] (5) at (0, -3) {};
    \end{pgfonlayer}
    \begin{pgfonlayer}{edgelayer}
    	\draw (1) to (i1);
    	\draw (3) to (i2);
    	\draw (2) to (i1');
    	\draw (4) to (i2');
    	\draw (i1) to (f1);
    	\draw (i2) to (f2);
    	\draw (i1') to (f3);
    	\draw (i2') to (f4);
    	\draw [bend right] (f1) to (U1);
    	\draw [bend left] (f2) to (U1);
    	\draw [bend right] (f3) to (U2);
    	\draw [bend left] (f4) to (U2);
    	\draw [bend right = 40] (U1) to (U3);
    	\draw [bend left = 40] (U2) to (U3);
    	\draw (U3) to (5);
    \end{pgfonlayer}
    \end{tikzpicture}
    \]\[
    \begin{tikzpicture}
    \begin{pgfonlayer}{nodelayer}
    	\node [style=construct] (i1) at (-1.5, 2) {$i_1$};
    	\node [style=construct] (i2) at (-0.5, 2) {$i_2$};
    	\node [style=construct] (i1') at (0.5, 2) {$i_1$};
    	\node [style=construct] (i2') at (1.5, 2) {$i_2$};
    	\node [style=construct] (f1) at (-1.5, 1) {$\phi$};
    	\node [style=construct] (f2) at (-0.5, 1) {$\phi$};
    	\node [style=construct] (f3) at (0.5, 1) {$\phi$};
    	\node [style=construct] (f4) at (1.5, 1) {$\phi$};
    	\node [style=construct] (U1) at (0, 0) {$\cup$};
    	\node [style=construct] (U2) at (-0.75, -0.75) {$\cup$};
    	\node [style=construct] (U3) at (0.375, -1.75) {$\cup$};
    	\node [style=none] (1) at (-1.5, 3) {};
    	\node [style=none] (2) at (-0.5, 3) {};
    	\node [style=none] (3) at (0.5, 3) {};
    	\node [style=none] (4) at (1.5, 3) {};
    	\node [style=none] (5) at (0.375, -2.75) {};
    	\node [style=none] (6) at (-1.5, -0.25) {};
    	\node [style=none] (7) at (0, -0.25) {};
    	\node [style=none] (8) at (1.5, -1) {};
    	\node [style=none] (9) at (-0.75, -1) {};
    	\node [style=none] () at (2.5, 0) {=};
    	\node [style=none] () at (2.5, 0.4) {(\ref{prop13})};
    	\node [style=none] () at (-2.5, 0) {=};
    	\node [style=none] () at (3.3, 0) {}; 
    \end{pgfonlayer}
    \begin{pgfonlayer}{edgelayer}
        \draw (1) to (i1);
        \draw (2) to (i2);
        \draw (3) to (i1');
        \draw (4) to (i2');
        \draw (i1) to (f1);
        \draw (i2) to (f2);
        \draw (i1') to (f3);
        \draw (i2') to (f4);
        \draw (f1) to (6.center);
        \draw [bend right] (f2) to (U1);
        \draw [bend left] (f3) to (U1);
        \draw [bend right] (6.center) to (U2);
        \draw (U1) to (7.center);
        \draw [bend left] (7.center) to (U2);
        \draw (U2) to (9.center);
        \draw [bend right = 40] (9.center) to (U3);
        \draw (f4) to (8.center);
        \draw [bend left = 40] (8.center) to (U3);
        \draw (U3) to (5);
    \end{pgfonlayer}
    \end{tikzpicture}
    \begin{tikzpicture}
	\begin{pgfonlayer}{nodelayer}
		\node [style=construct] (3) at (-1.5, 0.75) {$i_1$};
		\node [style=construct] (4) at (0.5, 0.75) {$i_2$};
		\node [style=construct] (5) at (-0.5, 0.75) {$i_1$};
		\node [style=construct] (6) at (1.5, 0.75) {$i_2$};
		\node [style=construct] (17) at (-1.5, -0.15) {$\phi$};
		\node [style=construct] (18) at (-0.5, -0.15) {$\phi$};
		\node [style=construct] (19) at (0.5, -0.15) {$\phi$};
		\node [style=construct] (20) at (1.5, -0.15) {$\phi$};
		\node [style=construct] (21) at (0, -1) {$\cup$};
		\node [style=construct] (22) at (-0.75, -1.75) {$\cup$};
		\node [style=construct] (23) at (0.375, -2.5) {$\cup$};
		\node [style=none] (8) at (-1.5, 2) {};
		\node [style=none] (9) at (0.25, 1.25) {};
		\node [style=none] (11) at (-0.25, 1.25) {};
		\node [style=none] (12) at (1.5, 2) {};
		\node [style=none] (13) at (-0.25, 1.5) {};
		\node [style=none] (14) at (0.25, 1.5) {};
		\node [style=none] (15) at (-0.5, 2) {};
		\node [style=none] (16) at (0.5, 2) {};
		\node [style=none] (24) at (0.375, -3.5) {};
		\node [style=none] (25) at (1.5, -1.75) {};
		\node [style=none] (26) at (-1.5, -1) {};
		\node [style=none] () at (2.5, -0.7) {=};
		\node [style=none] () at (3.3, -0.7) {}; 
	\end{pgfonlayer}
	\begin{pgfonlayer}{edgelayer}
		\draw (8.center) to (3.center);
		\draw [bend left, looseness=1] (9.center) to (4.center);
		\draw [bend right, looseness=1] (11.center) to (5.center);
		\draw (12.center) to (6.center);
		\draw (11.center) to (14.center);
		\draw [bend right, looseness=1] (14.center) to (16.center);
		\draw [bend right, looseness=1] (15.center) to (13.center);
		\draw (13.center) to (9.center);
		\draw (3.center) to (17.center);
		\draw (5.center) to (18.center);
		\draw (4.center) to (19.center);
		\draw (6.center) to (20.center);
		\draw (17.center) to (26.center);
		\draw [bend right=45] (26.center) to (22.center);
	    \draw [bend right=45] (18.center) to (21.center);
	    \draw [bend left=45] (19.center) to (21.center);
	    \draw (20.center) to (25.center);
	    \draw [bend left=45] (21.center) to (22.center);
	    \draw [bend left=45] (25.center) to (23.center);
	    \draw [bend right=45] (22.center) to (23.center);
		\draw (23.center) to (24.center);
	\end{pgfonlayer}
    \end{tikzpicture}
    \begin{tikzpicture}
	\begin{pgfonlayer}{nodelayer}
		\node [style=construct] (3) at (-1.5, 0.75) {$i_1$};
		\node [style=construct] (4) at (0.5, 0.75) {$i_2$};
		\node [style=construct] (5) at (-0.5, 0.75) {$i_1$};
		\node [style=construct] (6) at (1.5, 0.75) {$i_2$};
		\node [style=construct] (17) at (-1.5, -0.15) {$\phi$};
		\node [style=construct] (18) at (-0.5, -0.15) {$\phi$};
		\node [style=construct] (19) at (0.5, -0.15) {$\phi$};
		\node [style=construct] (20) at (1.5, -0.15) {$\phi$};
		\node [style=construct] (21) at (-1, -1) {$\cup$};
		\node [style=construct] (22) at (1, -1) {$\cup$};
		\node [style=construct] (23) at (0, -2) {$\cup$};
		\node [style=none] (8) at (-1.5, 2) {};
		\node [style=none] (9) at (0.25, 1.25) {};
		\node [style=none] (11) at (-0.25, 1.25) {};
		\node [style=none] (12) at (1.5, 2) {};
		\node [style=none] (13) at (-0.25, 1.5) {};
		\node [style=none] (14) at (0.25, 1.5) {};
		\node [style=none] (15) at (-0.5, 2) {};
		\node [style=none] (16) at (0.5, 2) {};
		\node [style=none] (24) at (0, -3) {};
	\end{pgfonlayer}
	\begin{pgfonlayer}{edgelayer}
		\draw (8.center) to (3.center);
		\draw [bend left, looseness=1] (9.center) to (4.center);
		\draw [bend right, looseness=1] (11.center) to (5.center);
		\draw (12.center) to (6.center);
		\draw (11.center) to (14.center);
		\draw [bend right, looseness=1] (14.center) to (16.center);
		\draw [bend right, looseness=1] (15.center) to (13.center);
		\draw (13.center) to (9.center);
		\draw (3.center) to (17.center);
		\draw (5.center) to (18.center);
		\draw (4.center) to (19.center);
		\draw (6.center) to (20.center);
		\draw [bend right] (17.center) to (21.center);
		\draw [bend left] (18.center) to (21.center);
		\draw [bend right] (19.center) to (22.center);
		\draw [bend left] (20.center) to (22.center);
		\draw [bend right] (21.center) to (23.center);
		\draw [bend left] (22.center) to (23.center);
		\draw (23.center) to (24.center);
	\end{pgfonlayer}
    \end{tikzpicture}
    \]\[
    \begin{tikzpicture}
	\begin{pgfonlayer}{nodelayer}
		\node [style=construct] (3) at (-1.5, 0.75) {$i_1$};
		\node [style=construct] (4) at (0.5, 0.75) {$i_2$};
		\node [style=construct] (5) at (-0.5, 0.75) {$i_1$};
		\node [style=construct] (6) at (1.5, 0.75) {$i_2$};
		\node [style=construct] (17) at (-1, 0) {$\gamma$};
		\node [style=construct] (18) at (1, 0) {$\gamma$};
		\node [style=construct] (19) at (-1, -1) {$\phi$};
		\node [style=construct] (20) at (1, -1) {$\phi$};
		\node [style=construct] (21) at (0, -2) {$\cup$};
		\node [style=none] (22) at (0, -3) {};
		\node [style=none] (8) at (-1.5, 2) {};
		\node [style=none] (9) at (0.25, 1.25) {};
		\node [style=none] (11) at (-0.25, 1.25) {};
		\node [style=none] (12) at (1.5, 2) {};
		\node [style=none] (13) at (-0.25, 1.5) {};
		\node [style=none] (14) at (0.25, 1.5) {};
		\node [style=none] (15) at (-0.5, 2) {};
		\node [style=none] (16) at (0.5, 2) {};
		\node [style=none] () at (-2.5, -1) {=};
		\node [style=none] () at (-2.5, -0.6) {(\ref{phiishom})};
		\node [style=none] () at (2.5, -1) {=};
		\node [style=none] () at (2.5, -0.6) {(\ref{isarehoms})};
		\node [style=none] () at (3.3, -1) {};
	\end{pgfonlayer}
	\begin{pgfonlayer}{edgelayer}
		\draw (8.center) to (3.center);
		\draw [bend left, looseness=1] (9.center) to (4.center);
		\draw [bend right, looseness=1] (11.center) to (5.center);
		\draw (12.center) to (6.center);
		\draw (11.center) to (14.center);
		\draw [bend right, looseness=1] (14.center) to (16.center);
		\draw [bend right, looseness=1] (15.center) to (13.center);
		\draw (13.center) to (9.center);
		\draw [bend right] (3.center) to (17.center);
		\draw [bend left] (5.center) to (17.center);
		\draw [bend right] (4.center) to (18.center);
		\draw [bend left] (6.center) to (18.center);
		\draw (18.center) to (20.center);
		\draw [bend left](20.center) to (21.center);
		\draw (17.center) to (19.center);
		\draw [bend right](19.center) to (21.center);
		\draw (21.center) to (22.center);
	\end{pgfonlayer}
    \end{tikzpicture}
    \begin{tikzpicture}
	\begin{pgfonlayer}{nodelayer}
		\node [style=construct] (0) at (0, -2.7) {$\sqcup$};
		\node [style=construct] (1) at (-1, 0) {$\cup$};
		\node [style=construct] (2) at (1, 0) {$\cup$};
		\node [style=construct] (17) at (-1, -0.9) {$i_1$};
		\node [style=construct] (18) at (-1, -1.8) {$\phi$};
		\node [style=construct] (19) at (1, -0.9) {$i_2$};
		\node [style=construct] (20) at (1, -1.8) {$\phi$};
		\node [style=none] (3) at (-1.5, 0.25) {};
		\node [style=none] (4) at (0.5, 0.25) {};
		\node [style=none] (5) at (-0.5, 0.25) {};
		\node [style=none] (6) at (1.5, 0.25) {};
		\node [style=none] (7) at (0, -3.6) {};
		\node [style=none] (8) at (-1.5, 1.5) {};
		\node [style=none] (9) at (0.25, 0.75) {};
		\node [style=none] (11) at (-0.25, 0.75) {};
		\node [style=none] (12) at (1.5, 1.5) {};
		\node [style=none] (13) at (-0.25, 1) {};
		\node [style=none] (14) at (0.25, 1) {};
		\node [style=none] (15) at (-0.5, 1.5) {};
		\node [style=none] (16) at (0.5, 1.5) {};
		\node [style=none] () at (2.5, -1.6) {=};
		\node [style=none] () at (2.5, -1.2) {(\ref{def:disjoint})};
		\node [style=none] () at (3.3, 0) {};
	\end{pgfonlayer}
	\begin{pgfonlayer}{edgelayer}
		\draw [bend right=55, looseness=1.25] (3.center) to (1.center);
		\draw [bend right=55, looseness=1.25] (4.center) to (2.center);
		\draw [bend left=55, looseness=1.25] (5.center) to (1.center);
		\draw [bend left=55, looseness=1.25] (6.center) to (2.center);
		\draw (0.center) to (7.center);
		\draw (8.center) to (3.center);
		\draw [bend left, looseness=1] (9.center) to (4.center);
		\draw [bend right, looseness=1] (11.center) to (5.center);
		\draw (12.center) to (6.center);
		\draw (11.center) to (14.center);
		\draw [bend right, looseness=1] (14.center) to (16.center);
		\draw [bend right, looseness=1] (15.center) to (13.center);
		\draw (13.center) to (9.center);
		\draw (1.center) to (17.center);
		\draw (2.center) to (19.center);
		\draw (17.center) to (18.center);
		\draw (19.center) to (20.center);
		\draw [bend right] (18.center) to (0.center);
		\draw [bend left] (20.center) to (0.center);
	\end{pgfonlayer}
    \end{tikzpicture}
    \begin{tikzpicture}
	\begin{pgfonlayer}{nodelayer}
		\node [style=construct] (0) at (0, -1) {$\sqcup$};
		\node [style=construct] (1) at (-1, 0) {$\cup$};
		\node [style=construct] (2) at (1, 0) {$\cup$};
		\node [style=none] (3) at (-1.5, 0.25) {};
		\node [style=none] (4) at (0.5, 0.25) {};
		\node [style=none] (5) at (-0.5, 0.25) {};
		\node [style=none] (6) at (1.5, 0.25) {};
		\node [style=none] (7) at (0, -2) {};
		\node [style=none] (8) at (-1.5, 1.5) {};
		\node [style=none] (9) at (0.25, 0.75) {};
		\node [style=none] (11) at (-0.25, 0.75) {};
		\node [style=none] (12) at (1.5, 1.5) {};
		\node [style=none] (13) at (-0.25, 1) {};
		\node [style=none] (14) at (0.25, 1) {};
		\node [style=none] (15) at (-0.5, 1.5) {};
		\node [style=none] (16) at (0.5, 1.5) {};
	\end{pgfonlayer}
	\begin{pgfonlayer}{edgelayer}
		\draw [bend right=55, looseness=1.25] (3.center) to (1.center);
		\draw [bend right=55, looseness=1.25] (4.center) to (2.center);
		\draw [bend left=55, looseness=1.25] (5.center) to (1.center);
		\draw [bend left=55, looseness=1.25] (6.center) to (2.center);
		\draw [bend right = 55] (1.center) to (0.center);
		\draw [bend left = 55] (2.center) to (0.center);
		\draw (0.center) to (7.center);
		\draw (8.center) to (3.center);
		\draw [bend left, looseness=1] (9.center) to (4.center);
		\draw [bend right, looseness=1] (11.center) to (5.center);
		\draw (12.center) to (6.center);
		\draw (11.center) to (14.center);
		\draw [bend right, looseness=1] (14.center) to (16.center);
		\draw [bend right, looseness=1] (15.center) to (13.center);
		\draw (13.center) to (9.center);
	\end{pgfonlayer}
\end{tikzpicture}
    \]
    
    Let $\sigma \in S_m$ and $\tau \in S_n$. Then 
    \begin{align*}
        \GMV(\sigma + \tau)(g \sqcup h)
        &= \GMV(\sigma + \tau)\phi(i_1(g) \cup i_2(h))
        \\&= \GMV(\sigma)\phi(i_1(g)) \cup \GMV(\tau)\phi(i_2(h))
        \\&= \GMV(\sigma(g)) \sqcup \GMV(\tau(h)),
    \end{align*}
    so the following diagram commutes.
    \[
    \begin{tikzcd}
        \GMV(m) \times \GMV(n)
        \arrow[r, "\sqcup"]
        \arrow[d, "\GMV(\sigma) \times \GMV(\tau)", swap]
        &
        \GMV(m+n)
        \arrow[d, "\GMV(\sigma + \tau)"]
        \\
        \GMV(m) \times \GMV(n)
        \arrow[r, "\sqcup", swap]
        &
        \GMV(m+n)
    \end{tikzcd}
    \]
    Thus $\sqcup$ is a natural transformation.
\end{proof}

\begin{proof}[Proof of Theorem \ref{thm:main}]
    Checking the coherence conditions for $\sqcup$ to be a laxator is a straightforward computation.
    Let $f \maps M \to N$. Then define the natural transformation $f_\V \maps \GMV \to \G_{N,\V}$ with components $(f_\V)_n \maps \GMV(n) \to \G_{N,\V}(n)$ given by the universal property. Composition is clearly preserved.
\end{proof}

\begin{thm}
    The functor $\G_{-,\V}$ is left adjoint to $E \maps \NM_\V \to \V$ where $E(F) = F(\mathbf 2)$ for $(F,\Phi) \maps (\S, +) \to (\V, \times)$ a $\V$-network model.
\end{thm}
Because of this, we call $\G_{M,\V}$ the \define{free $\V$-valued network model on the monoid $M$} or the \define{free $\V$ network model on $M$}.

\begin{proof}
    By construction, $\G_{M,\V}(\2) = M$, so let the unit $\eta = 1_{1_\V} \maps 1_\V \to \G_{-,\V}(\2)$. 
    
    We use the universal property of $\GMV$ to construct the counit. We define a map $F(\2) \to F(\n)$ for each vertex in $\KG_{n,2}$, and a map $F(\2) \times F(\2) \to F(\n)$ for each edge in $\KG_{n,2}$. 
    
    If $i,j \leq n$, then $F((1\; i)(2\; j)) \maps F(n) \to F(n)$. If $e$ is the unit of the monoid $F (\mathbf{n - 2} )$, and $m \in F(\2)$, then $\Phi_{\2, \n - \2} (m, e) \in F(n)$. Define maps $c_{i,j} \maps F(\2) \to F(\n)$ by 
    \[
        c_{i, j} = F((1\; i)(2\; j)) (\Phi_{\2, \n - \2}(m, e)).
    \]
    The intuition here is that $m$ is a value on one edge of the graph, and $e$ is a graph with $n-2$ vertices and no edges. Then $\Phi(m,e)$ is the graph with $n$ vertices, and just one $m$-valued edge between vertices $1$ and $2$. Then the permutation $(1\; i)(2\; j)$ permutes this one-edge graph to put $m$ between vertex $i$ and vertex $j$. So the map $c_{i,j}$ places the one-edge monoid $M$ at the $i,j$-position in the $n$-vertex monoid.
    
    Define maps $c_{i,j,p,q} \maps F(\2) \times F(\2) \to F(\n)$ by $c_{i,j,p,q} (\m,\m') = c_{i,j} (\m) c_{p,q} (\m')$. The second gives a monoid homomorphism precisely because $(F,\Phi)$ is a network model. 
    
    Then we get a map $(\epsilon_F)_\n \maps \G_{F(\2),\V}(\n) \to F(\n)$ by universal property, which gives a monoidal natural transformation automatically. That these maps form the components of a natural transformation can be seen by a routine computation.
    
    Notice that 
    \begin{align*}
        (\epsilon \G_{-,\V})_M = \epsilon_{\G_{M,\V}} = 1_{\G_{M,\V}},\\
        (\G_{-,\V} \eta)_M = \G_{1_M,\V} = 1_{\G_{M,\V}},\\
        (E \epsilon)_F = E(\epsilon_F) = (\epsilon_F)_2 = 1_{F(2)},\\
        (\eta E)_F = \eta_{F(2)} = 1_{F(2)}.
    \end{align*}
    Thus, checking that the snake equations hold is routine. 
\end{proof}

\begin{expl}
\label{ex:oldnew}
    In $\CMon$, products and coproducts are isomorphic. In particular, for a commutative monoid $M$, $\G_{M,\CMon} \cong \G_M$.
\end{expl}

Note that this does not indicate that varietal network models completely encompass ordinary network models. If $M$ is a noncommutative monoid,then $\G_{M, \CMon}$ is not defined, but $\G_M$ is.

\section{Commitment Networks}
\label{sec:commitment}

The motivating example of network models in general is $\SG$, the network model of simple graphs. By Example \ref{ex:oldnew}, this network model is an example of the main construction of this paper, $\SG = \G_{\B,\CMon}$. The boolean monoid is not only an object in $\CMon$, it is also an object in $\GMon$, the variety of graphic monoids. Then we can consider the network models $\G_{\B,\Mon}$ and $\G_{\B,\GMon}$.

\begin{expl}
    Elements of the monoid $\G_{\B,\Mon}(n)$ are words $e_{p_1,q_1} \dots e_{p_k,q_k}$. These words are interpreted as graphs with edges that look like they were built with popsicle sticks, and if two edges lie directly on top of each other, they are identified.
    Besides that relation, you can stack edges as high as you want by placing them between different pairs of vertices, but sharing one vertex.  
\end{expl}

There are networks one could imagine building with this popsicle stick intuition which are not allowed by this formalism. For instance, consider a network with three nodes and an edge for each pair of nodes, each overlapping exactly one of its neighbors, forming an Escher-esque ever-ascending staircase. This sort of network is not allowed by the formalism, since networks are actually equivalence classes of words, where letters have a definite position relative to each other. This is an important feature for this network model as it is necessary to guarantee that the procedure in the following example is well-defined, giving an algebra of the related network operad. What this means in terms of popsicle stick intuition is that allowed networks are built by placing popsicle sticks one at a time. 

\begin{expl}
    Elements of the $\G_{\B,\GMon}(n)$ are similar to those in the previous example, except that they must obey the graphic identity, $xyx = xy$ for all $x,y \in \G_{\B,\GMon}(n)$. What this means in the graphical interpretation is that all edges can be identified with the lowest occuring instance of an edge on the same vertex pair. This means that these networks in reduced form have at most as many edges as the complete simple graph with the same number of edges. Essentially these networks are simple graphs with a partial order on the edges which respects disjointness of edges.
\end{expl}

The networks in the previous example have exactly what we need in a network model to realize networks of bounded degree as an algebra of a network operad.

\begin{expl}[\bf Networks of bounded degree, revisited]
    The \define{degree} of a vertex in a simple graph is the number of edges in the graph which contain that vertex. 
    For $k \in \N$, we say that a simple graph is \define{$k$-bounded} if all vertices have degree less than or equal to $k$. Then we can consider the set $B_k(n)$ of $k$-bounded simple graphs.
    We can define an action of $\G_{\B, \GMon}(n)$ on $B_k(n)$ in the following way. 
    Let $g = e_1 \dots e_l \in \G_{\B, \GMon}(n)$ and $h \in B_k(n)$. 
    Choose a graph $h' \in \G_{\B, \GMon}(n)$ which has the same edges as $h$. 
    Define $h_0 = h'$, then define $h_i = h_{i-1}e_i$ if that is $k$-bounded, else $h_i= h_{i-1}$. 
    Let $hg$ denote $h_l$, which is a $k$-bounded element of $\G_{\B, \GMon}(n)$. 
    Let $\G^k_{\B, \GMon}(n)$ denote the set of $k$-bounded elements of $\G_{\B, \GMon}(n)$. 
    There is a function $s \maps \G^k_{\B, \GMon}(n) \to B_k(n)$. 
    So we define $h g$ to be $s(h_l)$.
    This is independent of the choice of $h'$ and defines an action of $\G_{\B,\GMon}$ on $B_k(n)$.
    
    The networks in Question \ref{question} can be represented by simple graphs with vertex degrees bounded by $k$. Then $B_k(n)$ gives an algebra of the operad $\O_{\B,\GMon}$.
\end{expl}

This resolves the conflict encountered in Question \ref{question}. Ordinary network models could not record the order in which edges were added to a network, which was necessary to define a systematic way of attempting to add new connections to a network which has degree limitations on each vertex. 

\section*{Acknowledgements}

The author would like to thank John Baez, John Foley, Christina Vasilakopoulou, Daniel Cicala, Jade Master, Christian Williams, and the anonymous referees for indispensable conversations, corrections, and suggestions. 

This work was supported by CASCADE Subcontract 6G09-UCR as part of the DARPA Complex Adaptive System Composition and Design Environment project. 

\bibliographystyle{plain}
\bibliography{biblio.bib}

\end{document}